\documentclass[draft]{gtpart}     % Basic GT/GTM/AGT style
\usepackage{amsopn}                  %% for DeclareMathOperator
\usepackage{amsfonts}                %% \mathbb, mathfrac needs it
\usepackage[bf,small]{caption2}
\usepackage{graphicx}

\usepackage{calc}                      % reimplements the LATEX commands \setcounter,
\usepackage{mdwlist} %%compact itemize, enumerate
\usepackage[all]{xy}
\input xy

\title{Dendroidal weak 2-categories}

\author{Andor Luk\'acs}
\givenname{Andor}
\surname{Luk\'acs}
\address{Babes-Bolyai University\\
Faculty of Mathematics and Computer Science\\
Str. Mihail Kogalniceanu nr. 1\\
RO-400084 Cluj-Napoca, Romania}
\email{lukacs.andor@gmail.com}
\urladdr{}


\arxivreference{}
\arxivpassword{}

\volumenumber{}
\issuenumber{}
\publicationyear{}
\papernumber{}
\startpage{}
\endpage{}
\doi{}
\MR{}
\Zbl{}
\received{}
\revised{}
\accepted{}
\published{}
\publishedonline{}
\proposed{}
\seconded{}
\corresponding{}
\editor{}
\version{}

\newtheorem{thm}{Theorem}[section]
\newtheorem{lem}[thm]{Lemma}
\newtheorem{cor}[thm]{Corollary}
\newtheorem{prop}[thm]{Proposition}
\newtheorem{conj}[thm]{Conjecture}

\theoremstyle{definition}
\newtheorem{dfn}[thm]{Definition}
\newtheorem{exa}[thm]{Example}

\theoremstyle{remark}
\newtheorem{rem}[thm]{Remark}

\DeclareMathOperator{\colim}{colim}

\DeclareMathOperator{\Id}{Id}
 \DeclareMathOperator{\edges}{Edg}
\DeclareMathOperator{\ccor}{Cor} \DeclareMathOperator{\vvert}{Vert}
\DeclareMathOperator{\leaves}{Leaves}
\DeclareMathOperator{\inn}{in}

\DeclareMathOperator{\intedges}{InEdg}
\DeclareMathOperator{\id}{id} 
 
 \DeclareMathOperator{\hcN}{hcN} \DeclareMathOperator{\Symm}{Symm}
\DeclareMathOperator{\Sk}{Sk} \DeclareMathOperator{\coSk}{coSk}
\SelectTips{eu}{}
\newdir{ >}{{}*!/-5pt/@{>}} %for xypic, to hava a distance btw the object and the tail of type >->
\renewcommand{\SS}{\mathbb{S}}

\newcommand{\rpd}{\Omega^\pi}

\numberwithin{equation}{section}
\numberwithin{figure}{section}
 \newcommand{\hcNd}{\hcN_d}

 \newcommand{\Op}{\mathcal{O}p}
\newcommand{\Ctg}{\mathcal{C}tg} \newcommand{\biCat}{bi\mathcal{C}tg}
\newcommand{\Cat}{\mathcal{C}at} \newcommand{\wCat}{w\mathcal{C}at}
\newcommand{\dSets}{d\mathcal{S}ets}

\newcommand{\Sets}{\mathcal{S}ets} \newcommand{\E}{\mathcal{E}}
 
 \newcommand{\N}{\mathbb{N}}

\newcommand{\To}{\longrightarrow} \newcommand{\op}{{\scriptsize{\textnormal
      {op}}}} 

\renewcommand{\C}{\mathcal{C}}

\begin{document}
\begin{abstract}
  We discuss the dendroidal notion of weak higher categories introduced by
  Moerdijk and Weiss in \cite{moerdijkweiss1} and we prove that dendroidal weak
  2-categories are equivalent to bicategories.
\end{abstract}

\maketitle

\section{Introduction}

Weakened notions of categories are central in many branches of mathematics.
One would often like to form certain ``categories'' where the composition of
arrows is not strictly associative, but only up to some coherent higher cells
that should be part of the structure.  Important examples of such
structures in the literature are homotopy $n$-types. Roughly speaking, a
homotopy $n$-type in topological spaces is the equivalence class of a space
$X$ such that all the homotopy groups $\pi_k(X)$ are trivial when $k>n$. These
classes are taken with respect to weak homotopy. Describing algebraic models
for homotopy $n$-types is a classical problem in algebraic topology. For
$n=2,3$ the first such models were given by Whitehead and Mac Lane
\cite{whitehead, whitemac}. Following the influence of Grothendieck and
R. Brown, who emphasized that groupoids should provide the natural framework
for homotopy types, higher categorical algebraic models of homotopy $3$-types
were given and studied by Leroy \cite{leroy}, Joyal and Tierney \cite{joyalt},
Berger \cite{berger}.

If we work out the topological intuitions coming from the interplay between
spaces, maps and homotopies between maps, we arrive to the abstract notion of
bicategories, first defined by B\'enabou in \cite{benabou}. Bicategories are
structures that consist of 0-cells (objects), 1-cells (arrows) and 2-cells;
1-cells are composable but not strictly associatively, the failure of this is
measured by some natural 2-cells. We can iterate the process to arrive to a
definition of tricategories and so on, the $n$-th step of this process would
give us a notion of weak $n$-categories. The problem we encounter is that each
step we take for defining a one-level higher notion increases radically the
complexity of the necessary coherence conditions between the higher
dimensional cells.  These steps also diminish the intuition on the nature of
the coherence axioms. As a result, there exist a plethora of different
definitions of weak $n$-categories in the literature. Comparing the different
notions of weak $n$-categories is one of the main problems in higher category
theory. One of the issues can be formulated as follows. Intuitively, the
right place to compare two notions of weak $n$-categories would be inside a
weak $n+1$-category, but how do we decide which notion of weak
$n+1$-categories to use for this comparison?

To deal with these problems, one can consider stricter- or non-iterative
approaches to define weak $n$-categories. The idea is that the resulting
stricter notions should be enough to deal with the applications on the one
hand, and the slogan is that ``{\em weak $n$-categories are strictifiable up
  to some extent}'' on the other. Examples of this approach include Baez and
Dolan's notion of $(\infty,n)$-categories, Tamsamani categories, etc. We would
like to mention explicitly one such example, originating in the observation
that the category of categories embeds to simplicial sets, via the nerve
functor:
\[
N\colon\Cat\To s\Sets.
\]
Certain simplicial sets which are not in the image of the nerve functor behave
much like categories, except that ``composition'' of arrows is well defined
only up to some higher degree terms in that simplicial set.  A. Joyal in
\cite{joyal} called these simplicial sets quasi-categories, although the
notion was already introduced by Boardman and Vogt under the name of
restricted Kan complex in \cite{boardmanvogt}. A quasi-category is an
$(\infty,1)$-category in Baez and Dolan's sense, i.e.  all the degree 2- or
higher cells are invertible. An important fact about quasi-categories is that
they are exactly the fibrant objects in the Joyal model structure for the
category of simplicial sets.

The starting point of our investigation is that, there exist a generalisation
of simplicial sets, that is suitable to study operads in the context of
homotopy theory. On the one hand, operads (or rather coloured operads) can be
viewed as generalisations of categories, where we consider arrows that can
have multiple inputs as opposed to one. It is natural to ask wether there
exists a presheaf category that extends the category of operads in the same
way as simplicial sets extend the category of categories via the nerve
functor. The question was studied in the papers of Moerdijk and Weiss
\cite{moerdijkweiss1, moerdijkweiss2}: the category of dendroidal sets
satisfies the requirements and fits in a commutative diagram of categories
\[
\xymatrix{\Cat\ar[r]^-N \ar@{>->}[d] &s\Sets\ar@{>->}[d]\\
  \Op\ar[r]^-{N_d} & d\Sets}
\]
Since dendroidal sets are an extension of simplicial sets, suitable for
studying the homotopy theory of operads, the theory of dendroidal sets
inherits a lot of questions from the theory of simplicial sets. For example,
this extension allows us to consider quasi-operads in the category of
dendroidal sets, i.e. analogs of Joyal's quasi-categories. One can then ask
whether the Joyal model structure on the category of simplicial sets extends
to that of dendroidal sets in such a way, that the fibrant objects of this
model category are exactly the quasi-operads. Cisinski and Moerdijk in
\cite{cisinski} gave a positive answer to this question.  One nice feature of
dendroidal sets, observed in \cite{moerdijkweiss1}, is that they contribute to
the theory of higher categories with a new compact definition of weak
$n$-categories.

The aim of this paper is to study the dendroidal definition of weak
$n$-categories mentioned above in low degree. We restrict ourselves to the
cases $n=1$ and $n=2$. In the case of degree $2$, the corresponding classical
notion is bicategories. We prove that dendroidal weak $2$-categories are
equivalent to bicategories (even more, they are isomorphic as categories). 

Our paper is organised as follows: 

In Section \ref{section:2} we introduce dendroidal sets, with emphasis on the
necessary notions and terminology that will be used in the next Sections. The
symmetric monoidal structure on the category of dendroidal sets that makes the
definition of dendroidal weak $n$-categories possible is induced by the
dendroidal nerve functor and the Boardman-Vogt tensor product of coloured
operads. The dendroidal Grothendieck construction gives us a way to
systematically glue together dendroidal sets, and is an important ingredient
that will allow us to consider later dendroidal weak $n$-categories with any
set of objects. The weakening of the higher dimensional cells in the
dendroidal setting is done with the homotopy coherent dendroidal nerve
functor, that is also introduced in this Section.  

In Section \ref{section:3} we define dendroidal weak $n$-categories and prove
that they are $3$-coskeletal for every $n$. This is an important property of
dendroidal weak $n$-categories, since it implies that degree $0$, $1$ and $2$
completely determines dendroidal weak $n$-categories. 

In Section \ref{section:4} first we describe dendroidal weak
$1$-categories. The iterative definition of dendroidal weak $n$-categories
makes it then possible to discuss dendroidal weak $2$-categories.  We prove
that the quasi-category of dendroidal weak 2-categories (denoted by
$i^*(\wCat^2)$) has equivalent homotopy category to the category of
bicategories:
\setcounter{thm}{4}\setcounter{section}{4}
\begin{thm}
  The category of unbiased bicategories $u\biCat$ and $ho(i^*(\wCat^2))$ are
  isomorphic. Hence the category of classical bicategories is equivalent to
  the category of dendroidal weak 2-categories.
\end{thm}
\setcounter{thm}{0}\setcounter{section}{1}
The definition and basic properties of classical and unbiased bicategories are
recalled in
the Appendix.
\section{Dendroidal sets}\label{section:2}
\subsection{Terminology and basic facts about dendroidal
  sets}\label{termin:dsets}
Dendroidal sets generalise simplicial sets in a suitable way for studying the
homotopy theory of (coloured) operads and their algebras. They were introduced
in the papers of I. Moerdijk and I. Weiss \cite{moerdijkweiss1,
  moerdijkweiss2}. The idea behind the notion of dendroidal sets is that in
the same way as simplicial sets help us understanding categories via the nerve
functor, there should be an analogous notion for studying coloured operads as
a generalisation of categories. Our goal here is to write a self-contained
introduction to dendroidal sets, including all the terminology necessary for
the next Sections.

Let us start with the notion of trees. A {\em tree} is a finite non-planar
contractible graph with a distinguished leaf called {\em root}. A tree thus
has many planar representatives, when we draw a picture of it we actually pick
one. We will use the following terminology on trees: $\ccor_n$ denotes the
$n$-corolla, i.e. the tree with one vertex an $n+1$ leaves (one of these
leaves is the root), $\vvert(T)$ denotes the set of vertices of the tree $T$,
$\edges(T)$ denotes the set of edges of the tree $T$ and $\intedges(T)$
denotes the set of internal edges of the tree $T$. We will say that a vertex
$v\in\vvert(T)$ of a tree is an {\em outer vertex} if $v$ is adjacent to at
most one inner edge of $T$.  For example, on the following picture of a
(planar representative of a) tree $T$ we have  $\vvert(T) = \{u,v,w\}$,
$\edges(T)=\{a,b,c,d,e,f\}$, $\intedges(T)=\{c,b\}$. The vertices $u$ and $w$
are outer vertices.
 
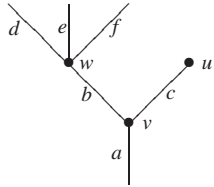
\[
  \xy<0.08cm,0cm>: 
  (40,10)*=0{}="5";
  (50,10)*=0{}="6"; 
  (60,10)*=0{}="7"; 
  (50,0)*=0{\bullet}="8";
  (70,0)*=0{\bullet}="9"; 
  (60,-10)*=0{\bullet}="10"; 
  (60,-20)*=0{}="11";
  (41,6)*{\scriptstyle d}; 
  (49,6)*{\scriptstyle e}; 
  (58,6)*{\scriptstyle f};
  (53,-5)*{\scriptstyle b}; 
  (67,-5)*{\scriptstyle c}; 
  (63,-10)*{\scriptstyle v}; 
  (58,-15)*{\scriptstyle a}; 
  (73,0)*{\scriptstyle u};
  (53,0)*{\scriptstyle w}; 
  "5";"8" **\dir{-}; 
  "6";"8" **\dir{-}; 
  "7";"8" **\dir{-}; 
  "8";"10" **\dir{-}; 
  "9";"10" **\dir{-}; 
  "10";"11" **\dir{-};
  \endxy
  \]

We will make frequent use of symmetric coloured operads (both in $\Sets$ and
enriched in a monoidal category $\E$) in the sense of \cite{leinster}  and we will refer to them as operads
from now on. Recall that if $P$ is an operad in $\Sets$, then it
comes equipped with a set of objects or colours $ob(P)$ and for each ordered
sequence of objects $\sigma=(e_1,\ldots,e_n;e)$, a set of operations
$P(e_1,\ldots,e_n;e)=P(\sigma)$. We will use the $\circ_i$-definition for the
composition of operations, i.e. if $\sigma$ is an ordered sequence or a {\em
  signature} as before and
$\rho=(f_1,\ldots,f_m;e_i)$ for a fixed $1\le i\le n$ then
\[
\sigma\circ_i\rho=(e_1,\ldots, e_{i-1},f_1,\ldots,f_m,e_{i+1},\ldots,e_n;e)
\]
and there is a given composition map
\[
\circ_i\colon P(\sigma)\times P(\rho)\To P(\sigma\circ_i\rho).
\]
The category of operads in $\Sets$ will be denoted by $\Op$, and the category
of planar- or non symmetric operads in $\Sets$ by $\Op^\pi$. Sometimes it will
be useful to construct operads from planar ones, via the free symmetrization
functor $\Symm\colon \Op^\pi\To \Op$, the left adjoint to the forgetful
functor $U\colon \Op\To\Op^\pi$.  This feature already appears in the
definition of dendroidal sets.

The category $\Omega^\pi$ consists of planar trees as objects and planar
operad maps as arrows. To be more precise, any planar tree $T$ gives rise to a
planar operad $\Omega^\pi(T)$. The objects of this non symmetric operad are
the edges of $T$, and the operations are freely generated by the vertices of
$T$, i.e. if $\sigma=(e_1,e_2,\ldots,e_n;e)$ is an ordered sequence of edges
of $T$ and there is a vertex $v$ with incoming edges $e_1,\ldots,e_n$ in this
order and outgoing edge $e$, then $\Omega^\pi(T)(\sigma)=\{v\}$. One can then
``compose'' vertices, indicated by the tree $T$. Hence a map $R\To T$ in
$\Omega^\pi$ is simply a planar operad map $\Omega^\pi(R)\To \Omega^\pi(T)$.
We observe that if $f\colon R\To T$ is an isomorphism, then the planar operad
structures imply that $R$ and $T$ have the same planar shape and they differ
only on the names of their edges and vertices. To avoid dealing with these
irrelevant isomorphisms, further on we will replace $\rpd$ by a skeleton of
it, and call this new category $\rpd$ as well. With this new convention, we
observe that all the maps of $\Omega^\pi$ are generated by two types, {\em
  faces} and {\em degeneracies}. These types of maps generalise the
corresponding notions in the category $\Delta$ defining simplicial sets, in
the following way.  Let $L_n$ denote the linear tree with $n$ vertices, $n\geq
0$:
\[
\xy<0.08cm,0cm>: (0,18)*=0{}="1"; (0,12)*=0{\bullet}="2";
(0,6)*=0{\bullet}="3"; (0,0)*=0{}="4"; (0,-1.7)*=0{\vdots}="5";
(0,-6)*=0{}="6"; (0,-12)*=0{\bullet}="7"; (0,-18)*=0{}="8"; "1";"2" **\dir{-};
"2";"3" **\dir{-}; "3";"4" **\dir{-}; "6";"7" **\dir{-}; "7";"8" **\dir{-};
\endxy
\]

If we consider the categorical definition of $\Delta$, we observe that the
category
\[ [n]=\xymatrix@1{0 &1\ar[l] &2\ar[l]&\cdots\ar[l]&n\ar[l]}
\] {\em is} in fact $[n]=\Omega(L_n)$, hence $\Delta$ is fully faithfully
embedded into $\rpd$ by $[n]\mapsto L_n$.

The face maps in $\Omega^\pi$ are all those monic operad maps $\partial\colon
\Omega^\pi(R)\To \Omega^\pi(T)$ which increase the number of vertices by one
(i.e.  $|\vvert(T)|=|\vvert(R)|+1$) and the degeneracies are all those epic
operad maps $\sigma\colon \Omega^\pi(T)\To \Omega^\pi(R)$ which decrease the
number of vertices by one.

It follows that face maps can be of the following types:
\begin{itemize*}
\item[(a)] the following picture is an example of an {\em outer face}
  
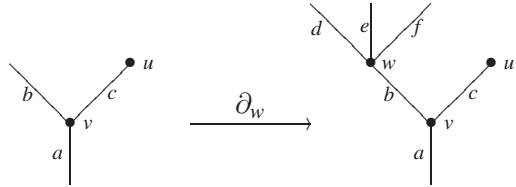
\[
  \xy<0.08cm,0cm>: (-10,0)*=0{}="1"; (10,0)*=0{\bullet}="2";
  (0,-10)*=0{\bullet}="3"; (0,-20)*=0{}="4"; (40,10)*=0{}="5";
  (50,10)*=0{}="6"; (60,10)*=0{}="7"; (50,0)*=0{\bullet}="8";
  (70,0)*=0{\bullet}="9"; (60,-10)*=0{\bullet}="10"; (60,-20)*=0{}="11";
  (13,0)*{\scriptstyle u}; (-7,-5)*{\scriptstyle b}; (7,-5)*{\scriptstyle c};
  (3,-10)*{\scriptstyle v}; (-2,-15)*{\scriptstyle a}; (41,6)*{\scriptstyle
    d}; (49,6)*{\scriptstyle e}; (58,6)*{\scriptstyle f};
  (53,-5)*{\scriptstyle b}; (67,-5)*{\scriptstyle c}; (63,-10)*{\scriptstyle
    v}; (58,-15)*{\scriptstyle a}; (73,0)*{\scriptstyle u};
  (53,0)*{\scriptstyle w}; (30,-7)*{\partial_w}; "1";"3" **\dir{-}; "2";"3"
  **\dir{-}; "3";"4" **\dir{-}; "5";"8" **\dir{-}; "6";"8" **\dir{-}; "7";"8"
  **\dir{-}; "8";"10" **\dir{-}; "9";"10" **\dir{-}; "10";"11" **\dir{-};
  {\ar(20, -10)*{};(40,-10)*{}};
  \endxy
  \]
  which is just an inclusion of operads;
\item[(b)] the following picture is an example of an {\em inner face}
  
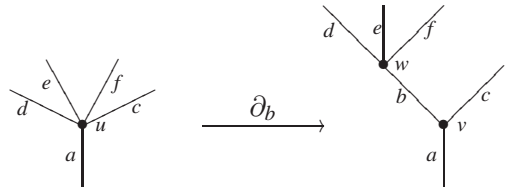
\[
  \xy<0.08cm,0cm>: (-12,-4)*=0{}="t1"; (-6,1)*=0{}="t2"; (6,1)*=0{}="t3";
  (12,-4)*=0{}="t4"; (0,-10)*=0{\bullet}="t5"; (0,-20)*=0{}="t6";
  (40,10)*=0{}="5"; (50,10)*=0{}="6"; (60,10)*=0{}="7";
  (50,0)*=0{\bullet}="8"; (70,0)*=0{}="9"; (60,-10)*=0{\bullet}="10";
  (60,-20)*=0{}="11"; (-10,-7)*{\scriptstyle d}; (-6,-3)*{\scriptstyle e};
  (6,-3)*{\scriptstyle f}; (9,-7)*{\scriptstyle c}; (-2,-15)*{\scriptstyle a};
  (3,-10)*{\scriptstyle u}; (41,6)*{\scriptstyle d}; (49,6)*{\scriptstyle e};
  (58,6)*{\scriptstyle f}; (53,-5)*{\scriptstyle b}; (67,-5)*{\scriptstyle c};
  (63,-10)*{\scriptstyle v}; (58,-15)*{\scriptstyle a}; (53,0)*{\scriptstyle
    w}; (30,-7)*{\partial_b}; "t1";"t5" **\dir{-}; "t2";"t5" **\dir{-};
  "t3";"t5" **\dir{-}; "t4";"t5" **\dir{-}; "t5";"t6" **\dir{-}; "5";"8"
  **\dir{-}; "6";"8" **\dir{-}; "7";"8" **\dir{-}; "8";"10" **\dir{-};
  "9";"10" **\dir{-}; "10";"11" **\dir{-}; {\ar(20, -10)*{};(40,-10)*{}};
  \endxy
  \]
  where $\partial_b\colon \Omega^\pi(R)\To \Omega^\pi(T)$ is the identity on
  the objects (edges), and sends the operation $u\in \Omega^\pi(R)(d,e,f,c;a)$
  to the composite operation $v\circ_1 w\in \Omega^\pi(T)(d,e,f,c;a)$, which
  we can denote without ambiguity by $v\circ_b w$.
\end{itemize*}
Note that the seemingly special cases of face maps into the corolla $\ccor_n$,
$n\ge 2$ fall under case (a): these face maps are all the $n+1$ possible edge
inclusions of the trivial tree $|$ to $\ccor_n$.

On the other hand, a degeneracy always looks like
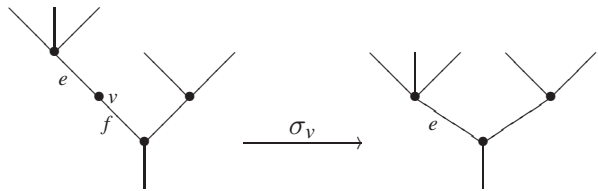
\[
\xy<0.06cm,0cm>: (0,40)*=0{}="1"; (10,40)*=0{}="2"; (20,40)*=0{}="3";
(10,30)*=0{\bullet}="4"; (20,20)*=0{\bullet}="5"; (30,10)*=0{\bullet}="6";
(40,20)*=0{\bullet}="7"; (30,30)*=0{}="8"; (50,30)*=0{}="9"; (30,0)*=0{}="10";
(80,30)*=0{}="11"; (90,30)*=0{}="12"; (100,30)*=0{}="13";
(90,20)*=0{\bullet}="14"; (105,10)*=0{\bullet}="15";
(120,20)*=0{\bullet}="16"; (110,30)*=0{}="17"; (130,30)*=0{}="18";
(105,0)*=0{}="19"; (12,24)*{\scriptstyle e}; (22,14)*{\scriptstyle f};
(23,20)*{\scriptstyle v}; (94,14)*{\scriptstyle e}; (65,13)*{\sigma_v};
"1";"4" **\dir{-}; "2";"4" **\dir{-}; "3";"4" **\dir{-}; "4";"5" **\dir{-};
"5";"6" **\dir{-}; "6";"7" **\dir{-}; "6";"10" **\dir{-}; "7";"8" **\dir{-};
"7";"9" **\dir{-}; "11";"14" **\dir{-}; "12";"14" **\dir{-}; "13";"14"
**\dir{-}; "14";"15" **\dir{-}; "15";"16" **\dir{-}; "16";"17" **\dir{-};
"16";"18" **\dir{-}; "15";"19" **\dir{-}; {\ar(52, 10)*{};(78,10)*{}};
\endxy
\]
where both of the objects $e,f$ are sent to $e$, the operation $v$ to the
identity operation $\id_e$ and $\sigma_v$ is the identity elsewhere.

We will use the following terminology with respect to faces and degeneracies:
\begin{itemize*}
\item If $e$ is an inner edge of a tree $T$, then $T/e$ denotes the tree
  resulting from $T$ by contracting $e$. The inner face corresponding to this
  contraction is usually denoted by $\partial_e\colon T/e\To T$.
\item If $v$ is an outer vertex of a tree $T$ (that is, it has exactly one
  inner edge adjacent to it), then $T/v$ denotes the tree resulting from $T$
  by removing $v$ and all the external edges adjacent to it (with the obvious
  choice for the root of $T/v$ when one of these external edges happens to be
  the root of $T$). We call this procedure ``cutting vertex $v$''. The outer
  face correspondig to cutting $v$ is usually denoted by $\partial_v\colon
  T/v\To T$.
\item If $v$ is a vertex of valence one of a tree $T$ then $T\backslash v$
  denotes the tree resulting from $T$ by removing $v$. The degeneracy
  corresponding to this removal is usually denoted by $\sigma_v\colon T\To
  T\backslash v$.
\end{itemize*}

The category $\Omega$ is obtained from $\Omega^\pi$ via the functor
$\Symm$. The objects of $\Omega$ are (non planar) trees and the arrows $R\To
T$ are operad maps $\Symm(\Omega(\bar R))\To \Symm(\Omega (\bar T))$, where
$\bar T$ denotes a planar representative of $T$. One can check that the
resulting operad does not depend on the chosen representatives, hence the
definition makes sense. Later on we will use this independence from chosen
representatives: we often describe the operad $\Omega(T)$ by picking a
representative $\bar T$ and giving only the description of the planar operad
$\Omega^\pi(\bar T)$.

The definition given above implies that there is an extra type of generator
for the maps in $\Omega$, namely the isomorphisms.

The category of dendroidal sets is the presheaf category on $\Omega$:
\[
\dSets:=\Sets^{\Omega^\op}=\textnormal{Funct}(\Omega^\op,\Sets).
\]
If $X$ is a dendroidal set, the elements of $X_T$ are called {\em dendrices of
  shape $T$}. The {\em representable dendroidal set} associated to a tree $T$
is the functor
\[
\Omega[T]:=\Omega(-,T)\colon\Omega^\op\To \Sets.
\]
By the Yoneda lemma, a dendrex $t\in X_T$ is the same thing as a map of
dendroidal sets $\Omega[T]\To X$. The Yoneda lemma in general is a very useful
tool in the theory of simplicial- and dendroidal sets, allowing us to swap
between maps and dendrices whenever needed. We are going to exploit this
property in the following without mentioning it. A first application of the
Yoneda lemma in the dendroidal setting proves that every dendroidal set is a
colimit of representable ones, a property generalising the well known fact for
simplicial sets.

For any given tree $T$ one can define some dendroidal subsets of the
representable $\Omega[T]$, like the boundary $\partial \Omega[T]$ or the inner
horn $\Lambda^e[T]$ with respect to the inner edge $e$. Dendroidal sets are
analogous to simplicial sets in many ways. For example, inner horns can be
used to define {\em inner Kan complexes} in the category of dendroidal sets:
we say that a dendroidal set $X$ {\em satisfies the inner Kan condition} if
for any inner horn $h\colon \Lambda^e[T]\To X$ there exists a dendrex $t\colon
\Omega[T]\To X$ such that the following diagram commutes:
\[
\xymatrix{ \Lambda^e[T]\ar[r]^-h\ar@{>->}[d] &X\\
  \Omega[T]\ar[ru]_-t }
\]
In this case $X$ is called an {\em inner Kan complex} or a {\em quasi-operad},
analogously to the simplicial case where an inner Kan complex was called by
A. Joyal a {\em quasi-category}.

Another notion that generalises from simplicial sets and categories to
dendroidal sets and operads is the nerve functor. The {\em dendroidal nerve}
$N_d\colon \Op\To \dSets$ can be defined by setting for any operad $P$
\[
\big(N_d(P)\big)_T:=\Op(\Omega(T),P).
\]
In the next few lines we introduce the notions of {\em $k$-skeleton} and {\em
  $k$-coskeleton} of a dendroidal set.  For every $k\in \N$ let $\Omega_k$
denote the full subcategory of $\Omega$ consisting of trees with at most $k$
vertices. The presheaf category on $\Omega_k$ is called the category of {\em
  $k$-truncated dendroidal sets} and is denoted by $\dSets_k$. The inclusion
$i_k\colon\Omega_k\To \Omega$ induces the truncation functor
$i_k^*\colon\dSets\To \dSets_k$ which has a left adjoint $(i_k)_!$ and a right
adjoint $(i_k)_*$. It follows that the composites
\[
(i_k)_!i_k^*,(i_k)_*i_k^*\colon \dSets \To \dSets
\]
form an adjoint pair of endofunctors. The left adjoint $(i_k)_!i_k^*$ is
usually denoted by $\Sk_k$ and is called the $k$-skeleton functor. The right
adjoint is denoted by $\coSk_k$ and is called the $k$-coskeleton functor.

A dendroidal set $X$ is said to be {\em $k$-coskeletal} if the canonical map
$X\To \coSk_k(X)$ is an isomorphism.  Another way to state this is that for
every dendroidal set $Y$ and every map of dendroidal sets $\phi\colon
\Sk_k(Y)\To X$ there exists a unique extension
\[
\xymatrix{\Sk_k(Y)\ar[r]^-\phi\ar@{>->}[d] & X\\ Y\ar@{.>}[ru]_-{\exists!}}
\]
Since any $Y\in\dSets$ is a colimit of representables, we can infer that if
the previous statement holds for all $Y=\Omega[T]$ where $T$ is any tree with
$k+1$ vertices, then it holds in general. In this case
$\Sk_k(Y)=\Sk_k(\Omega[T])=\partial \Omega[T]$. Note that in view of the
Yoneda lemma we can think of the composite
\[
\xymatrix{Sk_k(\Omega[T])\ar@{>->}[r]&\Omega[T]\ar[r]^-t& X}
\]
as the boundary- or $k$-skeleton of the dendrex $t$.  To emphasize this point
of view, sometimes we will denote this composite by $\Sk_k(t)$.

One can define the dual notion of $k$-skeletal dendroidal sets similarly.
\begin{rem}
  Note that the dendroidal definition of the functors $\Sk_k$ and $\coSk_k$
  uses a filtration of the objects of $\Omega$ by the number of the vertices
  of trees.  Later on, we will use the term {\em degree} to refer to this
  natural number.
\end{rem}

\subsection{A closed symmetric monoidal category structure on dendroidal sets}
Since $\dSets$ is a presheaf category, it can be endowed with the usual
cartesian closed category structure present in any presheaf category. There is
another interesting symmetric monoidal structure on $\dSets$ that will prove
to be useful in the definition of dendroidal weak $n$-categories of Section
\ref{section:3}. Our goal is to recall this monoidal structure in the current
section, together with those properties that will be used. For more details on
this subject one can consult \cite{moerdijkweiss1, moerdijkweiss2, weiss}.

One way to define the mentioned monoidal structure on $\dSets$ is by
transferring the Boardman-Vogt monoidal structure on $\Op$, via the dendroidal
nerve functor. We adopt this road, and we start by recalling the Boardman-Vogt
tensor product for symmetric operads (a generalisation of the B-V tensor
product for classical operads in \cite{boardmanvogt}).

Let $P, Q\in \Op$. We define a new operad, $P\otimes Q$, as follows. The set
of objects is $ob(P\otimes Q):=ob(P)\times ob(Q)$ and we denote the elements
of this set by $ a\otimes x:=(a,x)$. We describe the operations of $P\otimes
Q$ in terms of generators and relations. There are two types of generators:
\begin{itemize*}
\item[(a)] For any $p\in P(a_1,\ldots,a_n;a)$ and any $x\in ob(Q)$,
  \[
  p\otimes x\in P\otimes Q(a_1\otimes x,\ldots,a_n\otimes x;a\otimes x).
  \]
\item[(b)] For any $a\in ob(P)$ and any $q\in Q(x_1,\ldots,x_m;x)$,
  \[
  a\otimes q\in P\otimes Q(a\otimes x_1,\ldots,a\otimes x_m;a\otimes x).
  \]
\end{itemize*}
The relations also are of two types:
\begin{itemize*}
\item[(a)] Relations that imply precisely that the obvious maps
  \[
  \xymatrix@1{P\ar[r]^-{\id\otimes x} &P\otimes Q} \text{ and }
  \xymatrix@1{Q\ar[r]^-{a\otimes \id} &P\otimes Q}
  \]
  are maps of operads for any fixed $x\in ob(P), a\in ob(Q)$.
\item[(b)] For any $p\in P(a_1,\ldots,a_n;a)$ and $q\in Q(x_1,\ldots,x_m;x)$
  the following two operations are the same in $P\otimes Q$
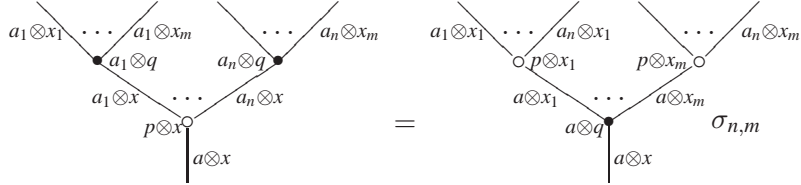
  \[
  \xy<0.08cm,0cm>: (0,0)*{ \xy<0.08cm,0cm>: (0,0)*=0{}="1";
    (0,10)*{\circ}="2"; (-15,20)*=0{\bullet}="3"; (15,20)*=0{\bullet}="4";
    (-25,30)*=0{}="5"; (-5,30)*=0{}="6"; (5,30)*=0{}="7"; (25,30)*=0{}="8";
    "1";"2" **\dir{-}; "2";"3" **\dir{-}; "2";"4" **\dir{-}; "3";"5"
    **\dir{-}; "3";"6" **\dir{-}; "4";"7" **\dir{-}; "4";"8" **\dir{-};
    (4,4)*{\scriptstyle a\otimes x}; (-4,9)*{\scriptstyle p\otimes x};
    (-12,14)*{\scriptstyle a_1\otimes x}; (0,14)*{ \ldots};
    (12,14)*{\scriptstyle a_n\otimes x}; (-9,20)*{\scriptstyle a_1\otimes q};
    (9,20)*{\scriptstyle a_n\otimes q}; (-25,25)*{\scriptstyle a_1\otimes
      x_1}; (-15,25)*{\ldots}; (-4,25)*{\scriptstyle a_1\otimes x_m};
    (15,25)*{\ldots}; (27,25)*{\scriptstyle a_n\otimes x_m};
    \endxy
  }; (70,0)*{ \xy<0.08cm,0cm>: (0,0)*=0{}="1"; (0,10)*=0{\bullet}="2";
    (-15,20)*{\circ}="3"; (15,20)*{\circ}="4"; (-25,30)*=0{}="5";
    (-5,30)*=0{}="6"; (5,30)*=0{}="7"; (25,30)*=0{}="8"; "1";"2" **\dir{-};
    "2";"3" **\dir{-}; "2";"4" **\dir{-}; "3";"5" **\dir{-}; "3";"6"
    **\dir{-}; "4";"7" **\dir{-}; "4";"8" **\dir{-}; (4,4)*{\scriptstyle
      a\otimes x}; (-4,9)*{\scriptstyle a\otimes q}; (-12,14)*{\scriptstyle
      a\otimes x_1}; (0,14)*{ \ldots}; (12,14)*{\scriptstyle a\otimes x_m};
    (-9,20)*{\scriptstyle p\otimes x_1}; (9,20)*{\scriptstyle p\otimes x_m};
    (-25,25)*{\scriptstyle a_1\otimes x_1}; (-15,25)*{\ldots};
    (-4,25)*{\scriptstyle a_n\otimes x_1}; (15,25)*{\ldots};
    (27,25)*{\scriptstyle a_n\otimes x_m};
    \endxy
  }; (35,-5)*{=}; (90,-5)*{\sigma_{n,m}};
  \endxy
  \]
  where $\sigma_{n,m}\in \Sigma_{n\cdot m}$ denotes the permutation that makes
  the order of the inputs on the right-hand side of the equation the same as
  the order of the inputs on the left-hand side.
\end{itemize*}
The tensor product we defined is a bifunctor $-\otimes -\colon\Op\times \Op\To
\Op$ and it induces a symmetric closed monoidal category structure on
$\Op$. The right adjoint of any functor $-\otimes Q$ is denoted by
$\underline{\Op}(Q,-)\colon \Op\To \Op$. In particular, $\underline
{\Op}(Q,\Sets)$ is {\em the operad of $Q$-algebras}. (For the definition of
$\underline {\Op}(Q,-)$ see \cite{weiss}.)

We can now make use of the functor $ N_d\colon \Op\To \dSets$ to transfer the
Boardman-Vogt tensor product to dendroidal sets:
\begin{itemize*}
\item[--] For any two representable dendroidal sets $\Omega[T]$ and
  $\Omega[R]$, define
  \[
  \Omega[T]\otimes \Omega[R]:= N_d(\Omega(T)\otimes \Omega(R)).
  \]
\item[--] Extend the definition cocontinuously, i.e. for any $X,Y\in \dSets$
  write $X=\colim_T\Omega[T]$, $Y=\colim_R\Omega[R]$ as colimits of
  representables and define
  \[
  X\otimes Y:=\colim_{T,R}\Omega[T]\otimes \Omega[R].
  \]
\end{itemize*}
The bifunctor $-\otimes-\colon \dSets\times\dSets\To \dSets$ induces a
symmetric closed monoidal structure on $\dSets$, the right adjoint of
$-\otimes Y$ is the functor
\[
\underline{\dSets}(Y,-)\colon \dSets\To \dSets,
\]
given on objects (by Yoneda lemma) by
\[
\underline{\dSets}(Y,Z)_T=\dSets(Y\otimes \Omega[T],Z).
\]
The following properties will prove to be useful in Section \ref{section:4}:
\begin{prop}(Lemma 4.3.3 in \cite{weiss}) For any operad $P\in \Op$ and for
  any tree $T\in \Omega$
  \[
  N_d(P)\otimes \Omega[T]\simeq N_d(P\otimes \Omega(T)).
  \]
\end{prop}
\begin{prop}(Corollary 9.3 in \cite{moerdijkweiss2}) For all operads
  $P,Q\in\Op$
  \[
  \underline{\dSets} (N_d(P),N_d(Q))\simeq N_d(\underline{\Op}(P,Q)).
  \]
\end{prop}

\subsection{The dendroidal Grothendieck construction}
The aim of this Section is to provide an ingredient we are going to use in the
description of dendroidal weak higher categories. The data we start with is a
functor $X\colon \SS^\op\To \dSets$ where $\SS$ is a cartesian category, and
we are going to assign to $X$ a new dendroidal set $\int_\SS X$, called {\em
  the Grothendieck construction of $X$}.

To achieve this goal, we need some preliminary definitions. Since $\SS$ is
cartesian it is an operad, hence it makes sense to talk about the dendroidal
set $N_d(\SS)$. Suppose that for a fixed tree $T$, $t\in N_d(\SS)$ is a
dendrex of shape $T$. That is, $t$ intuitively looks like the tree $T$
decorated with objects and operations of the operad $\SS$:
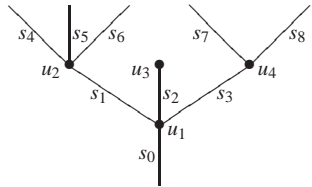
\[
\xy<0.08cm,0cm>: (0,0)*=0{}="1"; (10,0)*=0{}="2"; (20,0)*=0{}="3";
(30,0)*=0{}="4"; (50,0)*=0{}="5"; (10,-10)*=0{\bullet}="6";
(25,-10)*=0{\bullet}="7"; (40,-10)*=0{\bullet}="8"; (25,-20)*=0{\bullet}="9";
(25,-30)*=0{}="10"; (3,-5)*{\scriptstyle s_4}; (12,-5)*{\scriptstyle s_5};
(18,-5)*{\scriptstyle s_6}; (32,-5)*{\scriptstyle s_7}; (48,-5)*{\scriptstyle
  s_8}; (7,-11)*{\scriptstyle u_2}; (22,-11)*{\scriptstyle u_3};
(43,-11)*{\scriptstyle u_4}; (15,-15)*{\scriptstyle s_1};
(27,-15)*{\scriptstyle s_2}; (36,-15)*{\scriptstyle s_3};
(28,-21)*{\scriptstyle u_1}; (23,-25)*{\scriptstyle s_0}; "1";"6" **\dir{-};
"2";"6" **\dir{-}; "3";"6" **\dir{-}; "4";"8" **\dir{-}; "5";"8" **\dir{-};
"6";"9" **\dir{-}; "7";"9" **\dir{-}; "8";"9" **\dir{-}; "9";"10" **\dir{-};
\endxy
\]
where the $s_i$ are objects of $\SS$, and -- for example -- $u_1\colon
s_1\times s_2\times s_3\To s_0$ is a map in $\SS$. To such a $t$ we can assign
an object of $\SS$, called $\inn(t)$, which is the cartesian product of the
objects labeling the leaves of $T$: since $t\in \Op(\Omega(T),\SS)$,
\[
\inn(t):=\prod_{l\in\leaves(T)} t(l).
\]
Furthermore, we can assign to a $t\in N_d(\SS)_T$ and a map $\alpha\colon R\To
T$ of $\Omega$ a map in $\SS$
\[
\inn(\alpha)\colon \inn(t)\To \inn(\alpha^*t)
\]
by first composing the maps of $\SS$, indicated by $t$ and $\alpha$, and then
taking the product. For example, if $\alpha\colon R\To T$ is the inclusion to
the root vertex (in this case a composite of three outer faces)
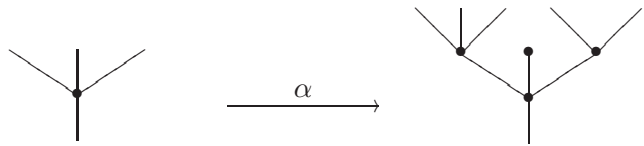
\[
\xy (0,0)*{ \xy<0.06cm,0cm>: (10,-10)*=0{}="6"; (25,-10)*=0{}="7";
  (40,-10)*=0{}="8"; (25,-20)*=0{\bullet}="9"; (25,-30)*=0{}="10"; "6";"9"
  **\dir{-}; "7";"9" **\dir{-}; "8";"9" **\dir{-}; "9";"10" **\dir{-};
  \endxy
}; (60,0)*{ \xy<0.06cm,0cm>: (0,0)*=0{}="1"; (10,0)*=0{}="2"; (20,0)*=0{}="3";
  (30,0)*=0{}="4"; (50,0)*=0{}="5"; (10,-10)*=0{\bullet}="6";
  (25,-10)*=0{\bullet}="7"; (40,-10)*=0{\bullet}="8";
  (25,-20)*=0{\bullet}="9"; (25,-30)*=0{}="10"; "1";"6" **\dir{-}; "2";"6"
  **\dir{-}; "3";"6" **\dir{-}; "4";"8" **\dir{-}; "5";"8" **\dir{-}; "6";"9"
  **\dir{-}; "7";"9" **\dir{-}; "8";"9" **\dir{-}; "9";"10" **\dir{-};
  \endxy
}; {\ar (20,-4)*{}; (40,-4)*{}}; (30,-2)*{\alpha};
\endxy
\]
and $t$ is as above, then $\alpha^*t$ is
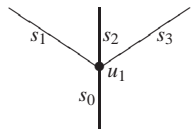
\[
\xy<0.08cm,0cm>: (10,0)*=0{}="6"; (25,0)*=0{}="7"; (40,0)*=0{}="8";
(25,-10)*=0{\bullet}="9"; (25,-20)*=0{}="10"; (15,-5)*{\scriptstyle s_1};
(27,-5)*{\scriptstyle s_2}; (36,-5)*{\scriptstyle s_3}; (28,-11)*{\scriptstyle
  u_1}; (23,-15)*{\scriptstyle s_0}; "6";"9" **\dir{-}; "7";"9" **\dir{-};
"8";"9" **\dir{-}; "9";"10" **\dir{-};
\endxy
\]
and $\inn(\alpha)=u_2\times u_3\times u_4$. In particular, if $\alpha$ is an
inner face or a degeneracy then $\inn(\alpha)$ is the identity map of
$\inn(t)=\inn(\alpha^*t)$, and if $\xymatrix@1{R'\ar[r]^-\beta
  &R\ar[r]^-\alpha &T}$ are maps of $\Omega$ then
$\inn(\alpha\beta)=\inn(\beta)\inn(\alpha)$.

In view of the definitions above we can define $\int_\SS X$ as follows. The
set $\left(\int_\SS X\right)_T$ consists of pairs $(t,x)$ where $t\in
N_d(\SS)_T$ and
\[
x\colon \Omega[T]\To \coprod_{s\in\SS} X(s)
\]
is a degree preserving map such that $x(r)\in X(\inn(r^*t))$ for any
$r\in\Omega[T]_R$. There is one more condition on $x$: it has to be compatible
with the dendroidal structure of the various dendroidal sets
involved. Explicitly, for a chain of arrows $\xymatrix@1{R'\ar[r]^-\alpha
  &R\ar[r]^- r&T}$ in $\Omega$ we have $r\in \Omega[T]_R$ and
$\alpha^*r=r\alpha\in\Omega[T]_{R'}$, hence
\[
x(r)\in X\big(\inn(r^*t)\big)_R \qquad\textnormal{and}\qquad x(\alpha^*r)\in
X\big(\inn((r\alpha)^*t)\big)_{R'}.
\]
The data above also induces two maps
\[
\xymatrix{X\big(\inn(r^*t)\big)_R\ar[rd]_-{\alpha^*} &&X\big(\inn((r\alpha)^*t)\big)_{R'}\ar[ld]^-{X(\inn\alpha)}\\
  &X\big(\inn(r^*t)\big)_{R'}}
\]
We require
\begin{equation}\label{groth1}
  \alpha^*\big(x(r)\big)=X(\inn\alpha)\big(x(\alpha^* r)\big).
\end{equation}
The dendroidal structure on $\int_\SS X$ is defined as follows. Suppose that
$\delta\colon R\To T$ is a map in $\Omega$ and $(t,x)\in \big(\int_\SS
X\big)_T$ a dendrex of shape $T$. The map $\delta$ induces the map of
dendroidal sets $\Omega[\delta]\colon\Omega[R]\To \Omega[T]$. We define
\begin{equation}\label{groth2}
  \delta^*(t,x):=\big(\delta^* t, x\circ\Omega[\delta]\big).
\end{equation}

One can check that with this structure $\int_\SS X$ is indeed a dendroidal
set. The following theorem and proposition collect two important properties of
the dendroidal Grothendieck construction.
\begin{thm}\label{ittay2}( \cite{moerdijkweiss1,weiss})
  Let $X\colon \SS^\op\To \dSets$ be a diagram of dendroidal sets. If for all
  $s\in\SS$ every $X(s)$ is an inner Kan complex then so is $\int_\SS X$.
  \end{thm}
\begin{prop}\label{grothcosk}
  Let $X\colon \SS^\op\To \dSets$ be a diagram of dendroidal sets and $k\ge 2$
  a natural number. If $X(s)$ is $k$-coskeletal for every $s\in \SS$ then so
  is $\int_\SS X$.
\end{prop}
\begin{proof}
  Let us start with the remark that $k\ge 2$ is needed because dendroidal
  nerves of operads are 2-coskeletal (a generalisation of the well known fact
  for nerves of categories, proven in \cite{moerdijkweiss1,weiss}).
 
  Our task is to prove that, for any tree $T$ with $k+1$ vertices, every map
  of dendroidal sets $\phi\colon\partial \Omega[T]\To \int_\SS X$ extends
  uniquely as
  \[
  \xymatrix{\partial \Omega[T]\ar@{>->}[d]\ar[r]^-\phi &\int_\SS X\\
    \Omega[T] \ar@{.>}[ru]_-{\exists !}}
  \]

  We suppose existence and prove uniqueness first. Let $(t_1,x_1),
  (t_2,x_2)\in \big(\int_\SS X\big)_T$ be two dendrices filling the boundary
  $\phi$. The dendroidal set $N_d(\SS)$ is $k$-co\-ske\-le\-tal since
  $k\ge2$. Hence by equation (\ref{groth2}) we can infer that $t_1=t_2=t$. Let
  $u\colon R\To T$ be a face. Since $u^*(t,x_1)=u^*(t,x_2)$, we obtain
  $x_1\circ\Omega[u]=x_2\circ\Omega[u]$. On the other hand,
  \[
  x_i(u)=\big(x_i\circ\Omega[u]\big)(\id_R)
  \]
  for $i=1,2$, implying $x_1(u)=x_2(u)$. We can use now equation
  (\ref{groth1}) for $r=\id_T$ and $\alpha=u$ to conclude that
  $u^*(x_1(\id_T))=u^*(x_2(\id_T))$ as dendrices of shape $R$ in $X(\inn
  (t))$. Since this is true for any face $u\colon R\To T$ and $X(\inn(t))$ is
  $k$-coskeletal, it follows that also $x_1(\id_T)=x_2(\id_T)$. We can infer
  that $x_1=x_2$, thus the filler is unique.

  The argument above also contains the information how to construct a filler
  $(t,x)$ of $\phi$, giving a proof of the existence of such an extension.
\end{proof}
\begin{rem}
  If we restrict our attention to dendroidal sets where the only nontrivial
  dendrices are of linear shapes, Proposition \ref{grothcosk} implies that the
  same property is true for simplicial sets and the simplicial Grothendieck
  construction.
\end{rem}
\subsection{The homotopy coherent dendroidal nerve of an
  operad}\label{homot_coh}
Let $\E$ be a symmetric monoidal model category with an {\em interval} $H$: that
is, an object $H$ of $\E$, together with two {\em points} $0\colon I\To H$ and
$1\colon I\To H$, an augmentation $\epsilon\colon H\To I$ and an associative binary
operation $\vee\colon H\otimes H\To H$ for which $0$ is unital and $1$ is absorbing, satisfying
$\epsilon 1 = \epsilon 0 = \id$.  In this case one can modify the nerve
construction for operads enriched in $\E$ in such a way that the resulting
dendroidal set encodes also the homotopies in the operad.

An interesting example of such a situation is when $\E$ is the category of
categories with the usual cartesian product, the folk model structure and the
interval $H$ is the category
\[
\xymatrix{0\ar[r]^-\simeq & 1}
\]
with two objects and one isomorphism between them. The required structure on
$H$ is the obvious one: $0$ is the neutral element, $1$ is the absorbing one,
and the rest of the interval structure on $H$ is completely determined by the
previous choices.  Indeed, since the unit of the monoidal structure is the
terminal object in $\E$ (the category $*$ with one object and no other
morphisms than the identity), the counit $\epsilon\colon H\To *$ is
obvious. The various compatibility conditions imply that the monoid structure
$\vee\colon H\times H\To H$ is given by ``the maximum operation'': on the
objects, $i\vee j=\max\{i,j\}$.

Since we are interested only in this example, from now on $\E$ denotes the
category of categories with the structure mentioned above, although everything
can be carried out similarly in the general case.

Let us denote the category of operads enriched in $\E$ by $\Op_\E$. The
functor
\[
\hcNd\colon \Op_\E\To\dSets
\]
is defined by
\[
\hcNd(P)_T:=\Op_\E(W(\Omega(T)),P)
\]
where $W$ is the $W$-construction for coloured operads (see
\cite{bergermoerdijk2} for details) and $\Omega(T)$ is the
discrete version in $\E$ of the operad induced by the tree $T$. We will need
later an explicit description of $W(\Omega(T))$, hence we give it here.

Recall that for a tree $T$
\[
\Omega(T)=\Symm(\Omega^\pi(\bar T))
\]
where $\Symm\colon\Op_\E^\pi \To \Op_\E$ is the $\E$-enriched version of the
symmetrization functor from non symmetric operads to operads, and $\bar T$ is
any planar representative of $T$. Moreover, the $W$-construction commutes with
$\Symm$, thus
\[
W\Omega(T)=\Symm\big(W\Omega^\pi(\bar T)\big).
\]
This property allows us to describe $W\Omega (T)$ by using an arbitrary planar
representative of $T$. The objects of $W \Omega^\pi(\bar T)$ are the edges of
$T$. Suppose that $\sigma=(e_1,e_2,\ldots,e_n;e)$ is an ordered sequence of
objects. We can distinguish two cases for the category of operations
corresponding to $\sigma$:
\begin{enumerate*}
\item[(1)] If $\Omega^\pi(\bar T)(\sigma)=\emptyset$ -- the empty category --
  then also $W\Omega^\pi (\bar T)(\sigma)=\emptyset.$
\item[(2)] If $\Omega^\pi(\bar T)(\sigma)\ne\emptyset$, it follows that $\bar
  T$ has a subtree $\bar T_\sigma$ with leaves $e_1,\ldots, e_n$ and root
  $e$. The set of internal edges of $\bar T_\sigma$ is denoted by $\intedges
  (\bar T_\sigma)$. From the $W$-construction it follows then that
  \[
  W\Omega^\pi (\bar T)(\sigma)=\prod_{f\in\intedges(\bar T_\sigma)} H,
  \]
  where in case the product is empty the result is the unit object of the
  monoidal structure, which is the category $*$ with one object and no other
  morphism than the identity.
\end{enumerate*}
We still need to define the composition maps in the operad $W\Omega^\pi(\bar
T)$. Suppose that $\sigma=(e_1,\ldots,e_n;e)$ and $\rho=(f_1,\ldots, f_m;e_i)$
are ordered sequences of edges of $T$, such that neither $\Omega^\pi(\bar
T)(\sigma)$, nor $\Omega^\pi(\bar T)(\rho)$ is the empty category. It follows
that $\bar T$ has subtrees $\bar T_\sigma$ and $\bar T_\rho$, the sets of
internal edges of these trees are disjoint and the tree $\bar
T_{\sigma\circ_i\rho}$ obtained by grafting along the edge $e_i$ has one more
internal edge than the previous two together. Let us denote these sets of
internal edges by $\textnormal{int}(\sigma)$, $\textnormal{int}(\rho)$ and
$\textnormal{int}(\sigma\circ_i\rho)$ respectively. The composition map
\[
\circ_i\colon W\Omega^\pi(\bar T)(\sigma)\times W\Omega^\pi(\bar T)(\rho)\To
W\Omega^\pi(\bar T)(\sigma\circ_i\rho)
\]
is given by
\[
\xymatrix{
  \left(\prod\limits_{\textnormal{int}(\sigma)}H\right)\times\left(\prod\limits_{\textnormal{int}(\rho)}H\right)\simeq
  \left(\prod\limits_{\textnormal{int}(\sigma)\cup
      \textnormal{int}(\rho)}H\right)\times *\ar[r]^-{\id\times 1} &
  \prod\limits_{\textnormal{int}(\sigma\circ_i\rho)}H},
\]
where the functor $1\colon *\To H$ is the absorbing element of $H$.

This concludes the description of the operad $W\Omega(T)$. Note that we still
need to mention how the dendroidal structure on $\hcNd(P)$ is defined. If
$\delta\colon R\To T$ is a face map in $\Omega$ then it induces a map of
operads $\delta\colon W\Omega(R)\To W\Omega(T)$ via the neutral element
functor $0\colon *\To H$. In case $\delta$ is a degeneracy, the induced
functor is obtained by the monoid structure $\vee\colon H\times H\To H$. These
definitions are functorial, hence they induce a dendroidal structure on
$\hcNd(P)$.
\section{Dendroidal weak $n$-categories} \label{section:3}
For any set $A$ there exists a planar operad $As^\pi_A$ whose algebras are
small categories with set of objects $A$. The objects of $As^\pi_A$ are
ordered pairs $(a_1,a_2)\in A\times A$, and the sets of operations are defined
by
\begin{align*}
  As^\pi_A\big(\quad;(a,a)\big) &= *,\\
  As^\pi_A\big((a_1,a_2),(a_2,a_3),\ldots,(a_{n-1},a_n);(a_1,a_n)\big) &=*
\end{align*}
and in all the other cases the set of operations is empty (those ordered
sequences $\sigma=(c_1,c_2, \ldots,c_n;c)$ of objects of $A\times A$ for which
$As^\pi_A(\sigma)$ is not empty will be called {\em admissible}).

Let $\alpha\colon As^\pi_A\To\Sets$ be a map of operads. The data-part of such
an $\alpha$ determines for any $(a_1,a_2)\in A\times A$ a set
$\mathcal{A}(a_1,a_2)$ and for any admissible signature
$\sigma=\big((a_1,a_2),(a_2,a_3),\ldots,(a_{n-1},a_n);(a_1,a_n)\big)$ a
function
\[
comp_\sigma\colon\mathcal{A}(a_1,a_2)\times
\mathcal{A}(a_2,a_3)\times\cdots\times \mathcal{A}(a_{n-1},a_n)\To
\mathcal{A}(a_1,a_n)
\]
which in the particular case of $\sigma=\big(\quad;(a,a)\big)$ is a function
$*\To \mathcal{A}(a,a)$.  The compatibility-part of such an $\alpha$ ensures
that the various functions $comp_\sigma$ fit nicely to define units and
compositions of arrows in a category $\mathcal{A}$ with object set
$A$. Indeed, we arrive to the conclusion that the relevant ordered sequences are of
the type $\big((a_1,a_2),(a_2,a_3);(a_1,a_3)\big)$ and
$\big(\quad;(a,a)\big)$, etc.

Since the forgetful functor $U\colon \Op\To \Op^\pi$ is right adjoint to the
symmetrization functor, we infer that the algebras of the operad $As_A:=
\Symm(As^\pi_A)$ are categories with set of objects $A$ as well.
\begin{rem}
  Note that in the description of $As_A^\pi$-algebras given above we used the
  unconventional ``left-to-right'' composition order for arrows, i.e.
  \[
  \xymatrix@1{\mathcal{A}(a_1,a_2)\times
    \mathcal{A}(a_2,a_3)\ar[r]&\mathcal{A}(a_1,a_3)}
  \]
  instead of the conventional
  \[
  \xymatrix@1{\mathcal{A}(a_2,a_3)\times \mathcal{A}(a_1,a_2)\ar[r]
    &\mathcal{A}(a_1,a_3)}
  \]
  To avoid unnecessary complications in the future, arising only from
  notation, whenever we need to give such a composition map associated to some
  signature $\sigma$, we will always stick to the order determined by
  $\sigma$, thus the unconventional order. However, when it is required to
  give extra details with explicit composites of maps, we will use the
  conventional ``right-to-left'' order.
\end{rem}

Let $X$ be a dendroidal set and define the functor
\[
\Cat(X)_-:\Sets^\op\To \dSets,\quad \Cat(X)_A:=\underline{\dSets}
(N_d(As_A),X).
\]
The dendroidal set of {\em categories enriched in $X$} (see \cite{moerdijkweiss1}) is by definition the
Grothendieck construction of $\Cat(X)_-$. We denote it by
\[
\Cat(X):=\int_{\Sets}\Cat(X)_-.
\]
One can iterate the process above to obtain a definition of the dendroidal set
of {\em $n$-categories enriched in $X$}:
\begin{align*}
  \Cat^0(X)&:= X,\\
  \Cat^n(X)&:=\Cat(\Cat^{n-1})(X).
\end{align*}
To see why this definition is plausible, one can try particular choices of
$X$. For example if $X=N_d(\Sets)$, we can prove inductively that $\Cat^n(X)$
is the dendroidal nerve of strict $n$-categories with the classical definition
(see also Example 4.5.5 in \cite{weiss}). Indeed, for $n=1$
\begin{align*}
  \Cat(N_d(\Sets))&=\int\limits_{A\in\Sets}\underline\dSets(N_d(As_A),N_d(\Sets))\\
  &\simeq \int\limits_{A\in\Sets} N_d(\underline\Op(As_A,\Sets))\\
  &\simeq \int\limits_{A\in\Sets} N_d(\mathcal{C}ateg_A)\\
  &\simeq N_d(\mathcal{C}ateg),
\end{align*}
where $\mathcal{C}ateg$ denotes the usual monoidal category of small
categories, viewed as an operad. The second part of the inductive proof is
similar (one uses that for any monoidal category $\mathcal{M}$, $\underline
\Op(As_A, \mathcal{M})\simeq \mathcal{C}ateg_A(\mathcal{M})$, where the
right-hand side denotes the monoidal category of categories enriched in
$\mathcal{M}$, with set of object $A$).

We are interested here in another choice for $X$, which yields the dendroidal
definition of weak $n$-categories: it is plausible to define $X:=\hcNd(\Ctg)$
where $\Ctg$ is the category of small categories enriched in $\E$. (Recall
from Section \ref{homot_coh} that  $\E$ denotes the
symmetric monoidal model category of categories, together with the interval
$H$. Hence the set of functors between two fixed categories is a category with
natural transformations as maps.)
\begin{dfn}\label{def:weak}(\cite{moerdijkweiss1})
  The dendroidal set of {\em weak $n$-categories} is defined as follows:
  \begin{align*}
    \wCat^0&:=N_d(\Sets),\\
    \wCat^n&:=\Cat^{n-1}(\hcNd(\Ctg)) \quad\textnormal{ for $n>0$}.
  \end{align*}
\end{dfn}
The rest of this Section is dedicated to the study of coskeletality of weak
$n$-categories.  The first result is
\begin{lem}\label{hcnd3}
  If\ \ $T\in\Omega$ is a tree with $3$ vertices and $t, s\in \wCat^1_T$
  satisfy $\Sk_2(t)=\Sk_2(s)$ then $t=s$.
\end{lem}
\begin{proof}
  To illustrate our argument better, we will work with a chosen tree, the
  general case can be carried out in the same way. So let $T\in\Omega$ be the
  tree with a planar representative as below.
  
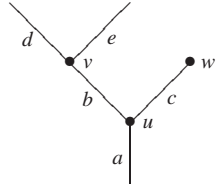
\[
  \xy<0.08cm,0cm>: (0,0)*=0{}="1"; (20,0)*=0{}="2"; (10,-10)*=0{\bullet}="3";
  (30,-10)*=0{\bullet}="4"; (20,-20)*=0{\bullet}="5"; (20,-30)*=0{}="6";
  (3,-6)*{\scriptstyle d}; (17,-6)*{\scriptstyle e}; (13,-16)*{\scriptstyle
    b}; (27,-16)*{\scriptstyle c}; (18,-26)*{\scriptstyle a};
  (23,-20)*{\scriptstyle u}; (13,-10)*{\scriptstyle v}; (33,-10)*{\scriptstyle
    w}; "1";"3" **\dir{-}; "2";"3" **\dir{-}; "3";"5" **\dir{-}; "4";"5"
  **\dir{-}; "5";"6" **\dir{-};
  \endxy
  \]
  Let $x\in \wCat^1_T$ be a dendrex of shape $T$, that is a map of operads
  enriched in $\E$, $x:W\Omega(T)\To \Ctg$.  Let us adopt the notations of
  Section \ref{homot_coh}. It follows that $x$ consists of compatible functors
  $x_\sigma:H^{\textnormal{int}(\sigma)}\To \Ctg(\sigma)$ and the only functor
  we have to describe in terms of the $2$-skeleton of $x$ is the one
  corresponding to $\sigma=(d,e;a)$ (the other functors lie in the image of
  $\Sk_2(x)$).  In this case the domain of $x_\sigma$ is the groupoid
  $H^2=H_c\times H_b$, represented as the square
  \[
  \xymatrix@+2mm{(0,0)\ar[r]^-{(\id_0,b)}\ar[d]_{(c,\id_0)} &(0,1)\ar@{.>}[d]^{(c,\id_1)}\\
    (1,0)\ar@{.>}[r]^-{(\id_1,b)} &(1,1)}
  \]
  where we think of the copy of $H$ corresponding to an internal edge $f$ as
  the groupoid $H_f=0\stackrel{f}{\To} 1$.  Since the domain is a groupoid, we
  observe that if $x_\sigma$ is already defined on a ``connected part'' of the
  ``square'' $H_c\times H_b$, then it is defined on the ``convex hull'' of
  that component. We conclude that in order to know $x_\sigma$, it is enough
  to know its image on the sets of arrows $Opd=\{(c,\id_1),(\id_1,b)\}$ and
  $Face=\{(c,\id_0), (\id_0,b)\}$.

  To conclude the proof, first we show that since $x$ is a map of operads,
  $x_\sigma(Opd)$ is determined by $\Sk_2(x)$.  Indeed, the commutative square
  \[
  \xymatrix{H_c\times \{v\}\ar[r]^-{x\times x}\ar[d]_{\circ_b} &\Ctg(b;a)\times \Ctg(d,e;b)\ar[d]^{\circ_b}\\
    H_c\times H_b\ar[r]^-{x_\sigma} &\Ctg(d,e;a)}
  \]
  implies that we know $x_\sigma$ on the arrow $(c,\id_1)$, and a similar
  square gives the image of $(\id_1,b)$.

  Second, we show that $x_\sigma(Face)$ is in the image of $\Sk_2(x)$. We
  observe that the inner faces $\partial_b:R\To T$ and $\partial_c:R'\To T$,
  according to the definition of the dendroidal structure on $\wCat^1$, induce
  enriched operad maps $W\Omega(R)\To W\Omega(T)$ and $W\Omega(R')\To
  W\Omega(T)$ respectively. Each of these maps has in its image the
  corresponding element of $Dend$, hence $x(Face)$ is in the image of
  $\Sk_2(x)$.
\end{proof}
\begin{prop}\label{hcndn}
  Let $T\in\Omega$ be a tree such that $|\vvert(T)|\ge3$. If $t, s\in
  \wCat^1_T$ satisfy $\Sk_2(t)=\Sk_2(s)$ then $t=s$.
\end{prop}
\begin{proof}
  We proceed by induction on $n=|\vvert(T)|$, the case $n=3$ is covered in
  Lemma \ref{hcnd3}. Suppose that $x:W\Omega(T) \To \Ctg$ is a dendrex of
  shape $T$. First we notice that we only need to describe the functor
  $x_\sigma:H^{{\textnormal {int}}(\sigma)}\To \Ctg(\sigma)$ in terms of the
  $\Sk_{n-1}(x)$ where $\sigma$ is the ordered sequence of colours
  $(\leaves(\bar{T});{\textnormal {root}}(\bar{T}))$ for a chosen planar
  representative $\bar T$. (The other components of $x$ are already contained
  in the image of $\Sk_{n-1}(x)$.)

  The domain of $x_\sigma$ is a groupoid with the shape of an $n$-cube, having
  in its vertices the trivial categories $(\epsilon_1,\ldots,\epsilon_n)$,
  $\epsilon_i\in\{0,1\}$. Denote by $\bar H_k$ the full subcategory of
  $H^{{\textnormal {int}}(\sigma)}$ spanned by the categories
  $(\epsilon_1,\ldots,\epsilon_{k-1},1,\epsilon_{k+1},\ldots,\epsilon_n)$,
  $\epsilon_i\in\{0,1\}$ (one of the hyperfaces of the $n$-cube, containing
  the vertex (1,1,\ldots,1)). Denote by $\phi_k$ the arrow $(0,0,\ldots,0)\To
  (0,\ldots,0,1,0,\ldots,0)$ of $H^{{\textnormal {int}}(\sigma)}$ (one of the
  edges of the $n$-cube, starting in $(0,0,\ldots,0)$). Define the sets
  \[
  Opd:=\{\bar H_k|k=1,2,\ldots n\} \quad \textnormal{and} \quad
  Face:=\{\phi_k|k=1,2,\ldots n\}.
  \]
  Since the ``convex hull'' of $Opd\cup Face$ is the whole domain of
  $x_\sigma$, it is enough to prove that $x_\sigma(Opd)$ is completely
  determined by $\Sk_{n-1}(x)$ and $x_\sigma(Face)$ is in the image of
  $\Sk_k(x)$. Both of these assertions are true, by similar arguments to the
  ones in the proof of Lemma \ref{hcnd3}.
\end{proof}
The following propositions of \cite{weiss} and \cite{moerdijkweiss1} helps us in proving that
$\wCat^1$ is 3-coskeletal.
\begin{prop}\label{ittay1}(Proposition 3.2.5 in \cite{weiss})
  Let $X$ be a dendroidal set and $k\ge 2$ an integer. If $X$ satisfies the
  strict inner Kan condition for all trees $T$ of degree at least $k$, then
  $X$ is $k$-coskeletal.
\end{prop}
\begin{prop}\label{ittay3} (Proposition 7.2 in \cite{moerdijkweiss1})
  Let $P$ be a locally fibrant operad in $\E$ (that is, for any ordered
  sequence of objects $\sigma=(c_1,\ldots,c_n;c)$ the category $P(\sigma)$ is
  fibrant with respect to the folk model structure). Then $\hcN_d(P)$ is an
  inner Kan complex.
\end{prop}
%%%%%%%%% \marginpar{again \cite{moerdijkweiss1}}
\begin{cor}\label{3coskel}(Lemma 4.6.3 in \cite{weiss}) The dendroidal set $\wCat^1$ is
  $3$-coskeletal.
\end{cor}
\begin{proof}
  In view of Proposition \ref{ittay1} it is enough to prove that $\wCat^1$
  satisfies the strict inner Kan condition for all trees with $|\vvert(T)|\ge
  3$. Let $T$ be such a tree. Theorem \ref{ittay3} implies that $\wCat^1$ is
  an inner Kan complex, hence every inner horn $\Lambda^e[T]\To\wCat^1$ has at
  least one filler $t$. Suppose that $s$ is an other filler of the same
  horn. Since $|\vvert(T)|\ge 3$, it follows that $\Sk_2(t)=\Sk_2(s)$. We
  conclude thus by Proposition \ref{hcndn} that $t=s$.
\end{proof}
Corollary \ref{3coskel} implies that $\wCat^n$ is $3$-coskeletal for every
$n\ge1$. (Note that $\wCat^0$ is already 2-coskeletal.) To prove this, the
following lemma is needed.
\begin{lem}\label{gencoskel}
  If $X$ is a $k$-coskeletal dendroidal set and $Z$ is an arbitrary dendroidal
  set then $\underline{\dSets}(Z,X)$ is $k$-coskeletal.
\end{lem}
\begin{proof}
  The goal is to see that for any dendroidal set $Y$ there exists a natural
  bijection
  \[
  \dSets(Y,\underline{\dSets}(Z,X))\simeq \dSets(\Sk_k
  Y,\underline{\dSets}(Z,X)).
  \]
  Indeed, once one observes that $\Sk_k(Y\otimes Z)\subseteq (\Sk_k Y)\otimes
  Z$, one can conclude that there are natural one-to-one correspondences
  between the following Hom sets:
  \begin{align*}
    \dSets(Y,\underline{\dSets}(Z,X))&\simeq \dSets(Y\otimes Z,X)\\
    &\simeq\dSets(Y\otimes Z,\coSk_k X)\\
    &\simeq \dSets(\Sk_k(Y\otimes Z),X)\\
    &\simeq \dSets((\Sk_k Y)\otimes Z,X)\\
    &\simeq \dSets(\Sk_k Y,\underline{\dSets}(Z,X)).
  \end{align*}
\end{proof}
\begin{thm}\label{gencoskel2}
  For every $n\ge 1$ the dendroidal set $\wCat^n$ is $3$-coskeletal.
\end{thm}
\begin{proof}
  We proceed by induction on $n$. It was proven in Corollary \ref{3coskel}
  that $\wCat^1$ is $3$-coskeletal. Suppose that $\wCat^n$ is
  $3$-coskeletal. It follows from Lemma \ref{gencoskel} that for any set $A$,
  the dendroidal set $\Cat(\wCat^n)_A=\underline{\dSets}(N_d(As_A),\wCat^n)$
  is $3$-coskeletal. Hence Proposition \ref{grothcosk} implies that
  \[
  \wCat^{n+1}=\int_{\Sets} \Cat(\wCat^n)_{-}
  \]
  is $3$-coskeletal.
\end{proof}
\section{Dendroidal weak 1- and 2-categories}\label{section:4}
\subsection{Weak 1-categories}\label{subsection:41}
We start the Section with the description of the dendroidal set
$\wCat^1=\hcNd(\Ctg)$. We can use Corollary \ref{3coskel} to come to the
conclusion that it is enough to describe the sets
$(\wCat^1)_T=\Op_\E(W\Omega(T), \Ctg)$ for trees $T$ with at most 3 vertices.
Before we start with the description, let us make a useful notational
convention: from now on, given $n$ categories $X_1,\ldots, X_n$ and integers
$1\le i\le j\le n$, $(X)_i^j$ will denote the category $X_i\times\cdots\times
X_j$.
\begin{enumerate*}
\item[(1)] The first choice of $T$ is the tree $|$. In this case
  $W\Omega(T)=\Omega(|)$ is the operad on one object and only the identity
  operation, hence an element of $(\wCat^1)_|$ is the same as the choice of a
  category.
\item[(2)] Let $T=\ccor_n$, the $n$-corolla. In this case still
  $W\Omega(\ccor_n)=\Omega(\ccor_n)$, hence an element of
  $(\wCat^1)_{\ccor_n}$ is the same as the choice of $n+1$ categories $X_1,
  \ldots,X_n$ and $X$, together with a functor $F:(X)_1^n\To X$. Note that in
  case $n=0$, the $\E$-enriched operad structure on $\Ctg$ implies that
  $(X)_1^n$ has to be considered the unit of the $\E$-enriched monoidal
  category $\Ctg$. This unit is the category $*$ on one object and no other
  arrows than the identity. Hence we infer that a dendrex of shape $\ccor_0$
  amounts to the choice of a category $X$, together with an object of it.
\item[(3)] Let $T=\ccor_n\circ_i\ccor_m$. Let us give a detailed description
  of maps of operads $\alpha:W\Omega(T)\To \Ctg$ since this is the first time
  when the interval $H$ plays a role in the definition of the operad
  $W\Omega(T)$.  So far it is clear that, as in cases (1) and (2), such an
  $\alpha$ determines
  \begin{enumerate*}
  \item[(3a)] a choice of $n+1$ categories $X_1,\ldots,X_n,X$ together with a
    functor $F_1:(X)_1^n\To X$;
  \item[(3b)] a choice of $m$ categories $Y_1,\ldots,Y_m$ and a functor
    $F_2:(Y)_1^m\To X_i$.
  \end{enumerate*}
  There is one more building part of such an $\alpha$, which is a functor
  \[
  H\To \Ctg\big((X)_1^{i-1}\times(Y)_1^m\times(X)_{i+1}^n,X\big).
  \]
  But such a functor contains exactly the same data as the choice of two
  functors $G,G':(X)_1^{i-1}\times (Y)_1^m\times (X)_{i+1}^n\To X$ and a
  natural isomorphism $\phi:G\To G'$.

  The only thing we have not covered yet with the investigation of such a
  dendrex is that $\alpha$ is a map of operads, which means that the diagram
  of categories
  \[
  \xymatrix{\ast\times \ast\ar[r]^-{\alpha\times\alpha}\ar[d]_{\circ_i}
    &\Ctg\big((X)_1^n,X\big)\times
    \Ctg\big((Y)_1^m,X_i\big)\ar[d]^{\circ_i}\\
    H\ar[r]^-\alpha
    &\Ctg\big((X)_1^{i-1}\times(Y)_1^m\times(X)_{i+1}^n,X\big)}
  \]
  is commutative. One can spell out that this yields to $G'=F_1\circ_i
  F_2$. We can conclude thus that the last bit of information $\alpha$
  provides is
  \begin{enumerate*}
  \item[(3c)] a choice of a functor $G:(X)_1^{i-1}\times (Y)_1^m\times
    (X)_{i+1}^n\To X$ and a natural isomorphism $\phi:G\To F_1\circ_iF_2$.
  \end{enumerate*}
\end{enumerate*}
For the remaining choices of the tree $T$ we give only the result.
\begin{enumerate*}
\item[(4)] Let $T=\ccor_n\circ_i(\ccor_m\circ_j\ccor_k)$. A map of operads
  $W\Omega(T)\To \Ctg$ is the same as
  \begin{enumerate*}
  \item[(4a)] a choice of $n+1$ categories $X_1,\ldots,X_n,X$ together with a
    functor $F_1:(X)_1^n\To X$;
  \item[(4b)] a choice of $m$ categories $Y_1,\ldots,Y_m$ and a functor
    $F_2:(Y)_1^m\To X_i$;
  \item[(4c)] a choice of $k$ categories $Z_1,\ldots,Z_k$ and a functor
    $F_3:(Z)_1^k\To Y_j$;
  \item[(4d)] a choice of a functor $G_1:(X)_1^{i-1}\times (Y)_1^m\times
    (X)_{i+1}^n\To X$ and a natural isomorphism $\phi_1:G_1\To F_1\circ_iF_2$;
  \item[(4e)] a choice of a functor $G_2:(Y)_1^{j-1}\times (Z)_1^k\times
    (Y)_{j+1}^m\To X_i$ and a natural isomorphism $\phi_2:G_2\To
    F_2\circ_jF_3$;
  \item[(4f)] a choice of a functor $K:(X)_1^{i-1}\times
    (Y)_1^{j-1}\times(Z)_1^k\times(Y)_{j+1}^m\times (X)_{i+1}^n\To X$ and two
    natural isomorphisms
    \[
    \psi_1:K\To F_1\circ_iG_2,\quad \psi_2:K\To G_1\circ_{\tilde j}F_3
    \]
    where $\tilde j=i+j-1$, such that the following diagram of natural
    isomorphisms is commutative:
    \[
    \xymatrix@C+5mm{K\ar[r]^-{\psi_1}\ar[d]_{\psi_2} &F_1\circ_i G_2\ar[d]^{F_1\circ_i \phi_2}\\
      G_1\circ_{\tilde i}F_3\ar[r]^-{\phi_1\circ_{\tilde j} F_3} &F_1\circ_i
      F_2\circ_j F_3}
    \]
  \end{enumerate*}
\item[(5)] Let $T=\ccor_n\circ_{i,j}(\ccor_m,\ccor_k)$ for some $1\le i<j\le
  n$. A map of operads $W\Omega(T)\To \Ctg$ is the same as
  \begin{enumerate*}
  \item[(5a)] a choice of $n+1$ categories $X_1,\ldots,X_n,X$ together with a
    functor $F_1:(X)_1^n\To X$;
  \item[(5b)] a choice of $m$ categories $Y_1,\ldots,Y_m$ and a functor
    $F_2:(Y)_1^m\To X_i$;
  \item[(5c)] a choice of $k$ categories $Z_1,\ldots,Z_k$ and a functor
    $F_3:(Z)_1^k\To X_j$;
  \item[(5d)] a choice of a functor $G_1:(X)_1^{i-1}\times (Y)_1^m\times
    (X)_{i+1}^n\To X$ and a natural isomorphism $\phi_1:G_1\To F_1\circ_iF_2$;
  \item[(5e)] a choice of a functor $G_2:(X)_1^{j-1}\times (Z)_1^k\times
    (X)_{j+1}^n\To X$ and a natural isomorphism $\phi_2:G_2\To F_2\circ_jF_3$;
  \item[(5f)] a choice of a functor $K:(X)_1^{i-1}\times
    (Y)_1^{m}\times(X)_{i+1}^{j-1}\times(Z)_1^k\times (X)_{j+1}^n\To X$ and
    two natural isomorphisms
    \[
    \psi_1:K\To F_1\circ_iG_2,\quad \psi_2:K\To G_1\circ_{\tilde j}F_3
    \]
    where $\tilde j=i+j-1$ (we suppose $j>i$), such that the following diagram
    of natural isomorphisms is commutative:
    \[
    \xymatrix@C+5mm{K\ar[r]^-{\psi_1}\ar[d]_{\psi_2} &F_1\circ_i G_2\ar[d]^{F_1\circ_i \phi_2}\\
      G_1\circ_{\tilde i}F_3\ar[r]^-{\phi_1\circ_{\tilde j} F_3}
      &F_1\circ_{i,j}(F_2, F_3)}
    \]
  \end{enumerate*}
\end{enumerate*}
We are going to illustrate with some examples the dendroidal structure of
$\wCat^1$ in the context described above.  Let $T=\ccor_n\circ_i\ccor_m$,
hence a dendrex $\alpha$ of shape $T$ is the same thing as the data described
in (3) above. If $\partial:\ccor_n\To T$ is the obvious outer face of $T$ then
$\partial^*(\alpha)$ corresponds to the choice of the categories
$X_1,\ldots,X_n, X$ and the functor $F_1:(X)_1^n\To X.$ If
$\partial:\ccor_{n+m-1}\To T$ is the inner face of $T$ then
$\partial^*(\alpha)$ corresponds to the choice of the categories
$X_1,\ldots,X_{i-1},Y_1,\ldots,Y_m,X_{i+1},\ldots,X_n$, $X$ and the functor
$G:(X)_1^{i-1}\times(Y)_1^m\times (X)_{i+1}^n\To X$. (The choice of $G$
instead of $F_1\circ_i F_2$ follows from the definition of the map of operads
$\partial\colon W\Omega(\ccor_{n+m-1})\To W\Omega(T)$.)

One can similarly decipher what a degeneracy looks like. A simple case of such
occurs when $R=\ccor_n\circ_i\ccor_1$, $T=\ccor_n$ and $\sigma:R\To T$ is the
degeneracy in question. If $\beta$ is a dendrex of shape $T$, that is a choice
of categories $X_1,\ldots,X_n,X$ and a functor $F_1:(X)_1^n\To X$, then
$\sigma^*(\beta)$ adds to the information contained in $\beta$ the identity
functor $\id:X_i\To X_i$, and the identity natural transformation $F_1\To
F_1\circ_i\id$.
\subsection{Weak $2$-categories}
We turn our attention now to the dendroidal set $\wCat^2$. Our goal is to
unpack the definition and compare the result with bicategories. It will become
apparent later that the right notion to compare the data of $\wCat^2$
contained in lower degrees is unbiased bicategories and their
homomorphisms. These notions were defined by Tom Leinster in
\cite{leinster}. We recall them in the Appendix.

The Section is organized as follows: First we analyse
the sets $\wCat_|^2$, $\wCat_{\ccor_1}^2$ and their relations to unbiased
bicategories, and we prove that the category of unbiased bicategories is
isomorphic to the homotopy category of dendroidal weak 2-categories. Then e
conclude the Section by a conjecture that predicts a stronger relation between
bicategories and dendroidal weak 2-categories.
\subsubsection{Dendroidal weak 2-categories}
In this Subsection we analyse those components of the dendroidal set $\wCat^2$
which will correspond to bicategories and homomorphisms of bicategories.
\subsubsection{Dendrices of shape $|$} \label{bicategories} Since
\[
(\wCat^2)_|=\left(\int_{\Sets}\Cat(\hcNd(\Ctg))_-\right)_|,
\]
the definition of the Grothendieck construction implies that an element of
$(\wCat^2)_|$ is a pair $(A,x)$, where $A$ is a set and $x$ is a dendrex of
shape $|$ in the dendroidal set $\Cat(\hcNd(\Ctg))_A$. Hence
\[
x\in\dSets(N_d(As_A)\otimes\Omega[|],\hcNd(\Ctg))=\dSets(N_d(As_A),\hcNd(\Ctg)).
\]
Since $\hcNd(\Ctg)$ is $3$-coskeletal, it is enough to look at the degree 0,
1, 2 and 3 components of $x$.
\begin{enumerate*}
\item[(0)] The degree 0 component of $x$ is the map of sets
  \[
  x_|\colon N_d(As_A)_|\To\hcNd(\Ctg)_|.
  \]
  Since $N_d(As_A)_|$ consists of the objects of the operad $As_A$ and
  $\hcNd(\Ctg)_|$ consists of categories, it follows that $x_|$ is the same
  thing as the choice of a category $\mathcal{A}(a_1,a_2)$ for each ordered
  pair $(a_1,a_2)\in A\times A$.
\item[(1)] Let us look at the $x_{\ccor_n}$ component, $n\in \N$. There are
  three cases to distinguish.

  First, an element in $N_d(As_A)_{\ccor_0}$ consists of a pair $(a,a)$ where
  $a\in A$, and the operation $\ast\in As_A\big(\quad ;(a,a)\big)$. We have
  seen in Subsection \ref{subsection:41} that an element of $\hcNd(\Ctg)_{\ccor_0}$ is a
  category together with an object of it. Since $x$ has to be compatible with
  the face map $|\To \ccor_0$, it follows that $x_{\ccor_0}$ picks for each
  $a\in A$ a functor $\Psi_a:*\To \mathcal{A}(a,a)$.

  Second, an element in $N_d(As_A)_{\ccor_1}$ consists of a pair $(a_1,a_2)\in
  A^2$ and the operation $\ast\in As_A\big((a_1,a_2);(a_1,a_2)\big)$, which is
  also the corresponding unit operation in the operad $As_A$. An element of
  $\hcNd(\Ctg)_{\ccor_1}$ is a functor between two chosen categories. Again,
  since $x$ has to be compatible with the various face and degeneracy maps, it
  follows that $x_{\ccor_1}$ amounts to choosing the identity functor on every
  already chosen category $\mathcal{A}(a_1,a_2)$. Hence $x_{\ccor_1}$ does not
  contribute with any new information.

  Third, for $n\ge 2$ $x_{\ccor_n}$ picks for each admissible ordered sequence
  \[
  \sigma=\big((a_1,a_2),(a_2,a_3),\ldots,(a_{n-1},a_n);(a_1,a_n)\big)
  \]
  of $n+1$ objects of $As_A$ a functor
  \[
  \Psi_\sigma:\mathcal{A}(a_1,a_2)\times\cdots\times
  \mathcal{A}(a_{n-1},a_n)\To \mathcal{A}(a_1,a_n).
  \]
  We can include the cases $n=0,1$ in the third one in the obvious way.
\item[(2)] The degree 2 component of $x$ consists of $x_T$ where
  $T={\ccor_n\circ_i\ccor_m}$ for the various $n,m,i\in \N$, $n\ne 0$.  There
  are face maps into the tree $T$ from the $m, n$ and $m+n-1$ corollas, and in
  case $n=1$ or $m=1$ there are also degeneracy maps with $T$ as the
  domain. Since $x$ has to be compatible with these faces and degeneracies, we
  can conclude that $x_T$ provides the following bit of extra data:

  For any pair of admissible ordered sequences
  \begin{align*}
    \sigma &=\big((a_1,a_2),\ldots,(a_{n-1},a_n);(a_1,a_n)\big),\\
    \rho &=\big((a_i,b_2),(b_2,
    b_3),\ldots,(b_{m-1},a_{i+1});(a_i,a_{i+1})\big)
  \end{align*}
  and any $1\leq i\leq n$ a natural isomorphism
  \[
  \phi_{\sigma,\rho,i}\colon\Psi_{\sigma\circ_i\rho}\To
  \Psi_{\sigma}\circ_i\Psi_{\rho},
  \]
  There is one condition on these natural isomorphisms: in case $n=1$ or
  $m=1$, the corresponding natural isomorphism has to be the identity (it
  follows from the compatibility with degeneracies again).
\item[(3)] The degree 3 components of $x$ do not give rise to any extra data,
  but the dendroidal identities with face maps induce relations on the already
  existing one, corresponding to cases (4f) and (5f) of Subsection
  \ref{subsection:41}. Explicitly, the diagrams of functors
  \[
  \xymatrix@C+2mm@R+2mm{\Psi_{\sigma\circ_i\rho\circ_j\tau}\ar[d]_-\phi
    \ar[r]^-\phi &\Psi_{\sigma\circ_i\rho}\circ_j \Psi_\tau\ar[d]^-{\phi\circ_j\Psi}\\
    \Psi_\sigma\circ_i\Psi_{\rho\circ_j\tau}\ar[r]^-{\Psi\circ_i\phi}
    &\Psi_\sigma\circ_i\Psi_\rho\circ_j\Psi_\tau}
  \]
  \[
  \xymatrix@C+2mm@R+2mm{\Psi_{\sigma\circ_{i,j}(\rho,\tau)}\ar[d]_-\phi
    \ar[r]^-\phi &\Psi_{\sigma\circ_i\rho}\circ_j \Psi_\tau\ar[d]^-{\phi\circ_j\Psi}\\
    \Psi_{\sigma\circ_j\tau}\circ_i\Psi_\rho\ar[r]^-{ \phi\circ_i\Psi}
    &\Psi_\sigma\circ_{i,j}(\Psi_\rho,\Psi_\tau)}
  \]
  are commutative.
\end{enumerate*}
\subsubsection{Dendrices of shape $\ccor_1$}\label{bicat maps}
An element of $(\wCat^2)_{\ccor_1}$ consists of pairs $(f,y)$ where $f:A\To B$
is a map of sets and
\[
y:\Omega[\ccor_1]\To\coprod_{S\in\Sets} \underline{\dSets}(N_d(As_S),\wCat^1)
\]
has three relevant components:
\begin{align*}
  y_A&\in\dSets(N_d(As_A),\wCat^1),\\
  y_B&\in\dSets(N_d(As_B),\wCat^1),\\
  y_f&\in\dSets(N_d(As_A\otimes \Omega(\ccor_1)),\wCat^1).
\end{align*}
These three components are related by the compatibility condition of the
Gro\-then\-dieck construction in the following way. Let $\ccor_1$ be
represented by the tree
\[
\xy<0.08cm,0cm>: (0,0)*=0{}="1"; (0,8)*=0{\bullet}="2"; (0,16)*=0{}="3";
(2,3)*{\scriptstyle 0}; (2,13)*{\scriptstyle 1}; (-3,8)*{\scriptstyle v};
"1";"2" **\dir{-}; "2";"3" **\dir{-};
\endxy
\]
thus the set of colours of the operad $\Omega(\ccor_1)$ is $\{0,1\}$ and the
only non-trivial operation is $ v\in\Omega(\ccor_1)(1;0)$.  If
$\partial_{1},\partial_{0}\colon|\To \ccor_1$ denote the face maps sending $|$
to the leaf and root of $\ccor_1$ respectively then $\partial_{1}^*(y_f)=y_A$
and $\partial_{0}^*(y_f)=f^*(y_B)$. The components $y_A$ and $y_B$ were
described in the first part of Subsection \ref{bicategories}, hence we need to
describe $y_f$ only.

Let us recall first the operad $As_A\otimes \Omega(\ccor_1)$ in more
detail. The set of colours of this operad contains all pairs $({\mathfrak
  a},l)$ where ${\mathfrak a}=(a_1,a_2)\in A^2$ is a colour of $As_A$ and
$l\in\{0,1\}$ is a colour of $\Omega(\ccor_1)$. The operations are generated
by the following three types of basic ones:
\[
\xy<0.08cm,0cm>: (0,0)*=0{}="1"; (0,8)*{\circ}="2"; (0,16)*=0{}="3";
(4,2)*{\scriptstyle{({\mathfrak a},0)}}; (4,14)*{\scriptstyle{({\mathfrak
      a},1)}}; (-5,8)*{\scriptstyle{({\mathfrak a},v)}}; "1";"2" **\dir{-};
"2";"3" **\dir{-};
\endxy
\]
is a picture of the basic operation
$(\mathfrak{a},v)\in\Omega(\ccor_1)(({\mathfrak a},1);({\mathfrak a},0))$
induced by $v\in\Omega(\ccor_1)(1;0)$ for any $\mathfrak{a}=(a_1,a_2)\in A^2$,
and
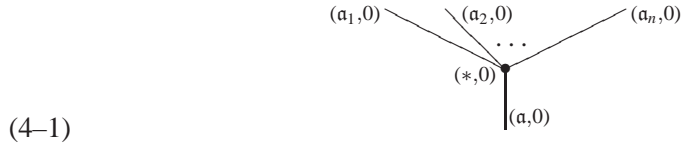
\begin{equation}\label{basicop0}
  \xy<0.08cm,0cm>:
  (0,0)*=0{}="1";
  (0,10)*=0{\bullet}="2";
  (-20,20)*=0{}="3";
  (-10,20)*=0{}="4";
  (20,20)*=0{}="5";
  (4,2)*{\scriptstyle{({\mathfrak a},0)}};
  (-25,19)*{\scriptstyle{({\mathfrak a}_1,0)}};
  (-3,19)*{\scriptstyle{({\mathfrak a}_2,0)}};
  (25,19)*{\scriptstyle{({\mathfrak a}_n,0)}};
  (-5,9)*{\scriptstyle{(*,0)}};
  (1,14)*{{\cdots}};
  "1";"2" **\dir{-};
  "2";"3" **\dir{-};
  "2";"4" **\dir{-};
  "2";"5" **\dir{-};
  \endxy
\end{equation}
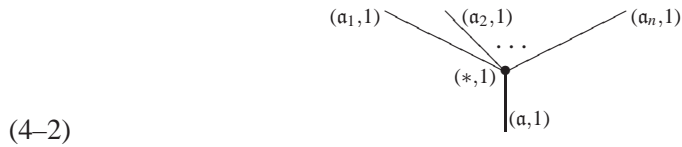
\begin{equation}\label{basicop1}
  \xy<0.08cm,0cm>:
  (0,0)*=0{}="1";
  (0,10)*=0{\bullet}="2";
  (-20,20)*=0{}="3";
  (-10,20)*=0{}="4";
  (20,20)*=0{}="5";
  (4,2)*{\scriptstyle{({\mathfrak a},1)}};
  (-25,19)*{\scriptstyle{({\mathfrak a}_1,1)}};
  (-3,19)*{\scriptstyle{({\mathfrak a}_2,1)}};
  (25,19)*{\scriptstyle{({\mathfrak a}_n,1)}};
  (-5,9)*{\scriptstyle{(*,1)}};
  (1,14)*{{\cdots}};
  "1";"2" **\dir{-};
  "2";"3" **\dir{-};
  "2";"4" **\dir{-};
  "2";"5" **\dir{-};
  \endxy
\end{equation}
are pictures of the basic operations in $As_A\otimes
\Omega(\ccor_1)((\mathfrak{a}_1,0),\ldots,(\mathfrak{a}_n,0);(\mathfrak{a},0))$
and $As_A\otimes
\Omega(\ccor_1)((\mathfrak{a}_1,1),\ldots,(\mathfrak{a}_n,1);(\mathfrak{a},1))$
respectively, induced by the unique operation
\[
*\in As_A((a_1,a_2),(a_2,a_3),\ldots,(a_n,a_{n+1});(a_1,a_{n+1)})
\]
where $\mathfrak{a}_i=(a_i,a_{i+1})\in A^2$ and $\mathfrak{a}=(a_1,a_{n+1})\in
A^2$. The operations generated this way are subject to the relations which
imply that the obvious projections $As_A\otimes \Omega(\ccor_1)\To As_A$,
$As_A\otimes \Omega(\ccor_1)\To \Omega(\ccor_1)$ are maps of operads, and to
the following relation (interchange law in the Boardman-Vogt tensor product
for operads):
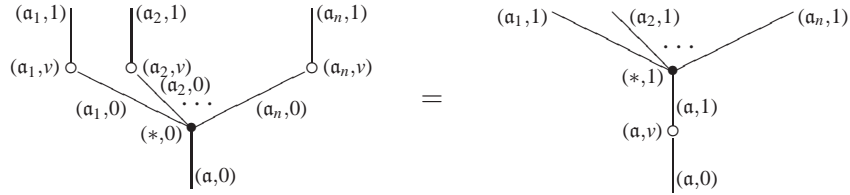
\[
\xy<0.08cm,0cm>: (0,0)*{ \xy<0.08cm,0cm>: (0,0)*=0{}="1";
  (0,10)*=0{\bullet}="2"; (-20,20)*{\circ}="3"; (-10,20)*{\circ}="4";
  (20,20)*{\circ}="5"; (-20,30)*=0{}="6"; (-10,30)*=0{}="7"; (20,30)*=0{}="8";
  (4,2)*{\scriptstyle{({\mathfrak a},0)}}; (-15,13)*{\scriptstyle{({\mathfrak
        a}_1,0)}}; (-1,17)*{\scriptstyle{({\mathfrak a}_2,0)}};
  (15,13)*{\scriptstyle{({\mathfrak a}_n,0)}};
  (-25,29)*{\scriptstyle{({\mathfrak a}_1,1)}};
  (-26,20)*{\scriptstyle{({\mathfrak a}_1,v)}};
  (-5,29)*{\scriptstyle{({\mathfrak a}_2,1)}};
  (-4,20)*{\scriptstyle{({\mathfrak a}_2,v)}};
  (25,29)*{\scriptstyle{({\mathfrak a}_n,1)}};
  (26,20)*{\scriptstyle{({\mathfrak a}_n,v)}}; (-5,9)*{\scriptstyle{(*,0)}};
  (1,14)*{{\cdots}}; "1";"2" **\dir{-}; "2";"3" **\dir{-}; "2";"4" **\dir{-};
  "2";"5" **\dir{-}; "3";"6" **\dir{-}; "4";"7" **\dir{-}; "5";"8" **\dir{-};
  \endxy
}; (40,0)*{=}; (80,0)*{ \xy<0.08cm,0cm>: (0,-10)*=0{}="0"; (0,0)*{\circ}="1";
  (0,10)*=0{\bullet}="2"; (-20,20)*=0{}="3"; (-10,20)*=0{}="4";
  (20,20)*=0{}="5"; (4,-8)*{\scriptstyle{({\mathfrak a},0)}};
  (4,4)*{\scriptstyle{({\mathfrak a},1)}}; (-25,19)*{\scriptstyle{({\mathfrak
        a}_1,1)}}; (-3,19)*{\scriptstyle{({\mathfrak a}_2,1)}};
  (25,19)*{\scriptstyle{({\mathfrak a}_n,1)}}; (-5,9)*{\scriptstyle{(*,1)}};
  (-5,0)*{\scriptstyle{(\mathfrak{a},v)}}; (1,14)*{{\cdots}}; "1";"0"
  **\dir{-}; "1";"2" **\dir{-}; "2";"3" **\dir{-}; "2";"4" **\dir{-}; "2";"5"
  **\dir{-};
  \endxy
};
\endxy
\]
The properties of the operad $As_A\otimes \Omega(\ccor_1)$ imply
\begin{lem} \label{wcat2lem} For any ordered sequence
  $\sigma=\big((\mathfrak{a}_1,l_1),\ldots,
  (\mathfrak{a}_n,l_n);(\mathfrak{a},l)\big)$, the corresponding set of
  operations $As_A\otimes \Omega(\ccor_1)(\sigma)$ contains at most one
  element. Moreover, in case $As_A\otimes \Omega(\ccor_1)(\sigma)$ is not
  empty, $\mathfrak a_i=(a_i,a_{i+1})$ and $\mathfrak a=(a_1,a_{n+1})$ for
  some $a_1,\ldots,a_{n+1}\in A$.
\end{lem}
\begin{proof}
  When $l=1$, the only possibilities for the sequence $\sigma$ that give
  nonempty $As_A\otimes \Omega(\ccor_1)(\sigma)$ are the ones corresponding to
  (\ref{basicop1}). Each such set of operations contains exactly one element:
  $(*,1)$. If $l=0$ and all $l_i=0$, then again the only nonempty sets of
  operations are the ones corresponding to (\ref{basicop0}), with a unique
  operation $(*,0)$ in each.

  The remaining cases to study are the ones when $\sigma$ is such that $l=0$
  and there exists $i$ with $l_i=1$. Suppose that in such a case $As_A\otimes
  \Omega(\ccor_1)(\sigma)$ is not empty. The interchange law implies then that
  these operations can always be reduced to a form, which can be visualized by
  a tree resulting from iterated gluing of the following two types of
  operations:
  
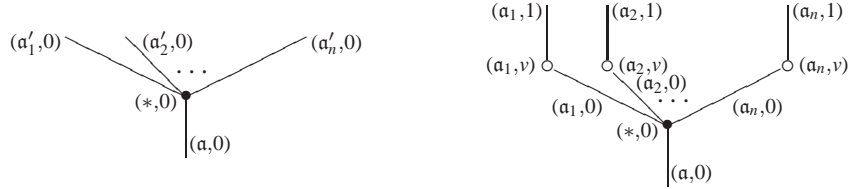
\[
  \xy<0.08cm,0cm>: (0,0)*{ \xy<0.08cm,0cm>: (0,0)*=0{}="1";
    (0,10)*=0{\bullet}="2"; (-20,20)*=0{}="3"; (-10,20)*=0{}="4";
    (20,20)*=0{}="5"; (4,2)*{\scriptstyle{({\mathfrak a},0)}};
    (-25,19)*{\scriptstyle{({\mathfrak a}'_1,0)}};
    (-3,19)*{\scriptstyle{({\mathfrak a}'_2,0)}};
    (25,19)*{\scriptstyle{({\mathfrak a}'_n,0)}};
    (-5,9)*{\scriptstyle{(*,0)}}; (1,14)*{{\cdots}}; "1";"2" **\dir{-};
    "2";"3" **\dir{-}; "2";"4" **\dir{-}; "2";"5" **\dir{-};
    \endxy
  }; (80,0)*{ \xy<0.08cm,0cm>: (0,0)*=0{}="1"; (0,10)*=0{\bullet}="2";
    (-20,20)*{\circ}="3"; (-10,20)*{\circ}="4"; (20,20)*{\circ}="5";
    (-20,30)*=0{}="6"; (-10,30)*=0{}="7"; (20,30)*=0{}="8";
    (4,2)*{\scriptstyle{({\mathfrak a},0)}};
    (-15,13)*{\scriptstyle{({\mathfrak a}_1,0)}};
    (-1,17)*{\scriptstyle{({\mathfrak a}_2,0)}};
    (15,13)*{\scriptstyle{({\mathfrak a}_n,0)}};
    (-25,29)*{\scriptstyle{({\mathfrak a}_1,1)}};
    (-26,20)*{\scriptstyle{({\mathfrak a}_1,v)}};
    (-5,29)*{\scriptstyle{({\mathfrak a}_2,1)}};
    (-4,20)*{\scriptstyle{({\mathfrak a}_2,v)}};
    (25,29)*{\scriptstyle{({\mathfrak a}_n,1)}};
    (26,20)*{\scriptstyle{({\mathfrak a}_n,v)}}; (-5,9)*{\scriptstyle{(*,0)}};
    (1,14)*{{\cdots}}; "1";"2" **\dir{-}; "2";"3" **\dir{-}; "2";"4"
    **\dir{-}; "2";"5" **\dir{-}; "3";"6" **\dir{-}; "4";"7" **\dir{-};
    "5";"8" **\dir{-};
    \endxy
  };
  \endxy
  \]
  where ``gluing'' means identifying the two roots indexed by
  $(\mathfrak{a},0)$. We can conclude that $\mathfrak a_i=(a_i, a_{i+1})$,
  $\mathfrak a=(a_1,a_{n+1})$ for some $a_1,\ldots,a_{n+1}\in A$ and
  $As_A\otimes \Omega(\ccor_1)(\sigma)$ contains again exactly one element.

  Let us illustrate the latter situation in a particular example, for instance
  when
  \[
  \sigma=\big((\mathfrak{a}_1,1), (\mathfrak{a}_2,0), (\mathfrak{a}_3,0),
  (\mathfrak{a}_4,1),(\mathfrak{a}_5,1); (\mathfrak{a},0)\big).
  \]
  Any operation in $As_A\otimes \Omega(\ccor_1)(\sigma)$ can be reduced to the
  form
  
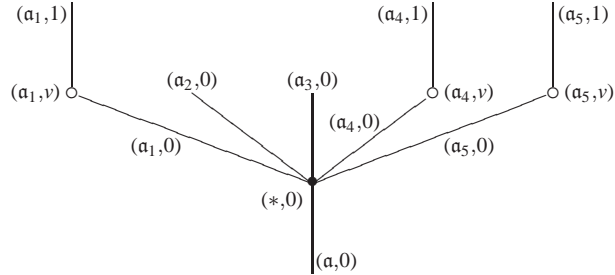
\[
  \xy<0.08cm,0cm>: (0,0)*=0{}="1"; (0,15)*=0{\bullet}="2";
  (-40,30)*{\circ}="3"; (-20,30)*=0{}="4"; (0,30)*=0{}="5";
  (20,30)*{\circ}="6"; (40,30)*{\circ}="7"; (-40,45)*=0{}="8";
  (20,45)*=0{}="9"; (40,45)*=0{}="10"; (4,2)*{\scriptstyle{({\mathfrak
        a},0)}}; (-45,43)*{\scriptstyle{({\mathfrak a}_1,1)}};
  (-26,21)*{\scriptstyle{({\mathfrak a}_1,0)}};
  (-20,32)*{\scriptstyle{({\mathfrak a}_2,0)}};
  (0,32)*{\scriptstyle{({\mathfrak a}_3,0)}};
  (15,43)*{\scriptstyle{({\mathfrak a}_4,1)}};
  (45,43)*{\scriptstyle{({\mathfrak a}_5,1)}};
  (7,25)*{\scriptstyle{({\mathfrak a}_4,0)}};
  (26,21)*{\scriptstyle{({\mathfrak a}_5,0)}};
  (-46,30)*{\scriptstyle{({\mathfrak a}_1,v)}};
  (26,30)*{\scriptstyle{({\mathfrak a}_4,v)}};
  (46,30)*{\scriptstyle{({\mathfrak a}_5,v)}}; (-5,12)*{\scriptstyle{(*,0)}};
  "1";"2" **\dir{-}; "2";"3" **\dir{-}; "2";"4" **\dir{-}; "2";"5" **\dir{-};
  "2";"6" **\dir{-}; "2";"7" **\dir{-}; "3";"8" **\dir{-}; "6";"9" **\dir{-};
  "7";"10" **\dir{-};
  \endxy
  \]
  Since the set of operations is assumed to be nonempty, $\mathfrak
  a_i=(a_i,a_{i+1})$ for all $1 \le i\le 6$ and $\mathfrak a= (a_1,a_6)$, and
  there is exactly one element in the set $As_A\otimes
  \Omega(\ccor_1)(\sigma)$ (that is the composite illustrated by the tree
  above).
\end{proof}

We can use this discussion about the operad $As_A\otimes \Omega(\ccor_1)$ to
understand the $y_f$ component of an element $y\in \wCat^2_{\ccor_1}$.  The
degree $0$ component of $y_f$ consists of the map of sets
\[(y_f)_|:N_d(As_A\otimes \Omega(\ccor_1))_| \To (\wCat^1)_|,
\]
thus $(y_f)_|$ amounts to choosing categories $\C_1(a_1,a_2)$ and
$\C_0(a_1,a_2)$ for every pair $(a_1,a_2)\in A^2$. The compatibility condition
in the Grothendieck construction implies that these categories are the
corresponding ones appearing in the definitions of $y_A$ and
$y_B$. Explicitly,
\begin{align*}
  \C_1(a_1,a_2)&=\mathcal{A}(a_1,a_2)\textnormal{ of $(y_A)_|$,}\\
  \C_0(a_1,a_2)&=\mathcal{B}(f(a_1),f(a_2))\textnormal{ of $(y_B)_|$.}
\end{align*}
Lemma \ref{wcat2lem} implies that $y_f$ in degree 1 amounts to the choice of a
functor
\[
\Psi_\sigma\colon\C_{l_1}(a_1,a_2)\times\cdots \times
\C_{l_{n-1}}(a_{n-1},a_n)\To \C_{l}(a_1,a_n)
\]
for every ordered sequence
\[
\sigma=\big(((a_1,a_2),l_1),((a_2,a_3),l_2),\ldots,((a_{n-1},a_n),l_{n-1});((a_1,a_n),l)\big)
\]
of objects of $As_A\otimes\Omega(\ccor_1)$. A particular case of such a
functor is
\[
\Psi_\sigma\colon \mathcal{A}(a_1,a_2)\To \mathcal{B}(f(a_1),f(a_2)).
\]
As a consequence of the compatibility conditions of the Grothendieck
construction, in case $l=l_i=1$ or $l=l_i=0$ for all $i$, one gets back the
corresponding functors $\Psi^A$ or $\Psi^B$ respectively, resulting from $y_A$
or $y_B$ (we described them in Subsection \ref{bicategories}).

The data contained in degree 2 of $y_f$ amounts to choices of natural
transformations $\phi\colon\Psi_\sigma\circ \Psi_\rho\To
\Psi_{\sigma\circ\rho}$, and similarly, the information contained in degree 3
can be described with the same semantics as in the description of $\wCat^2_|$.

\subsubsection{The relation between bicategories and dendroidal\\ weak
  2-categories}
In this Subsection we aim to establish a relation between dendroidal weak
2-categories and (unbiased) bicategories. Let us recall first that the
homotopy category of an inner Kan complex $X$ in simplicial sets is the
category $ho(X):=\tau(X)$, where $\tau$ is the left adjoint of the nerve
functor $N\colon\Cat\To s\Sets$. The category $ho (X)$ is defined as follows:
\begin {itemize*}
\item[--] the objects of $ho (X)$ are the elements of $X_0$;
\item[--] the arrows of $ho (X)$ are equivalence classes of elements of $X_1$,
  where the equivalence relation is left homotopy: if $a,b\in X_0$ and $f,g
  \colon a\To b$ are elements of $X(a,b)\subseteq X_1$, we say that $f$ is
  left homotopic to $g$ if there exists an element of $X_2$ that fills the
  triangle
  \[
  \xy<0.08cm,0cm>: (0,0)*+{\scriptstyle {a}}="2"; (30,0)*+{\scriptstyle
    {b}}="3"; (10,15)*+{\scriptstyle a}="1"; (15,-2)*{\scriptstyle g};
  (21,9)*{\scriptstyle f}; (3,9)*{\scriptstyle {1_a}}; {\ar@{=}"1";"2"};
  {\ar"2";"3"}; {\ar"1";"3"};
  \endxy
  \]
\end{itemize*}
\begin{rem}
  We could have chosen to define the equivalence relation above with the dual
  notion of right homotopy, since $X$ being an inner Kan complex implies that
  these two notions are the same. In fact, this is the only reason we renamed
  the category $\tau(X)$: the description of this category is much easier when
  $X$ is an inner Kan complex.
\end{rem}
The restriction functor $i^*\colon d\Sets\To s\Sets$ preserves inner Kan
complexes, hence it makes sense to talk about the category $ho (i^*(\wCat^2))$
(we will call it {\em the category of dendroidal weak 2-categories}). Our goal
is to compare this category with $u\biCat$.

The description of the elements of $\wCat^2_|$ in Subsection
\ref{bicategories} can be summarized in
\begin{prop}\label{bicatcomp0}
  The objects of the category $ho (i^*(\wCat^2))$ are in one-to-one
  correspondence with the objects of $u\biCat$.
\end{prop}
Let us turn to the comparison of morphisms of the categories in question. By
Proposition \ref{bicatcomp0} we have defined a functor $\Phi\colon u\biCat\To
ho (i^*(\wCat^2))$ on the objects. We complete the definition of $\Phi$ by
\begin{prop}\label{bicatcomp1}
  For any $\mathbb{A}, \mathbb{B}\in u\biCat$ there is a one-to-one
  correspondence between the hom-sets $u\biCat(\mathbb{A},\mathbb{B})$ and $ho
  (i^*(\wCat^2))(\Phi(\mathbb{A}),\Phi(\mathbb{B}))$. This correspondence is
  functorial.
\end{prop}
\begin{proof}
  Suppose that $(F,f)\colon \mathbb{A}\To\mathbb{B}$ is a map of unbiased
  bicategories (with strict unit). With the use of the description of
  $\wCat^2_{\ccor_1}$ we gave in Subsection \ref{bicategories} first we define
  an $y^F\in\wCat^2_{\ccor_1}$, induced by $(F,f)$. Recall, that such an $y^F$
  is determined by three components, and the components $y^F_A, y^F_B$ are
  obvious.

  Let us define the component $y^F_f\colon N_d(As_A)\otimes
  \Omega[\ccor_1]\To\wCat^1$. In degree $0$ again the definition of $y^F_f$ is
  obvious, the first non-trivial choices we have to make arise in degree 1.
  We need to give the components $(y^F_f)_{\ccor_n}$ for every $n\ge 1$,
  i.e. a functor
  \[
  \Psi_\sigma\colon\C_{l_1}(a_1,a_2)\times \C_{l_2}(a_2,a_3)\times\cdots\times
  \C_{l_{n-1}}(a_{n-1},a_n)\To \mathcal{B}(f(a_1),f(a_n))
  \]
  for every ordered sequence $\sigma$ with $a_i\in A$, $l_i\in\{0,1\}$, where
  \[
  \C_{l_i}(a_i,a_{i+1})=
  \begin{cases}
    \mathcal{A}(a_i,a_{i+1}) & \text{if} \quad l_i=1,\\
    \mathcal{B}(f(a_i),f(a_{i+1})) &\text{if} \quad l_i=0.
  \end{cases}
  \]
  Define this functor to be the composite
  \[
  \xymatrix@C+4mm{\C_{l_1}(a_1,a_2)\times\cdots\times \C_{l_{n-1}}(a_{n-1},a_n)\ar[d]_{G}\ar@{.>}[rd]\\
    \mathcal{B}(f(a_1),f(a_2))\times\cdots\times
    \mathcal{B}(f(a_{n-1}),f(a_n))\ar[r]_-{\Psi^B_\sigma} &
    \mathcal{B}(f(a_1),f(a_n))}
  \]
  where $G$ is the product of the functors
  $\xymatrix@1{\C_{l_i}(a_i,a_{i+1})\ar[r] &\mathcal{B}(f(a_i),f(a_{i+1}))}$
  that are either $F_{(a_i,a_{i+1})}$ or identity, depending on $l_i$. These
  choices define $y^F_f$ in degree 1.

  The next step is to give the components of $y^F_f$ in degree 2, that is we
  have to give natural isomorphisms $\Psi_\sigma\circ\Psi_\rho\To
  \Psi_{\sigma\circ\rho}$. It is a straightforward computation to see that for
  any such natural isomorphism we have exactly one choice: some pasting of the
  invertible 2-cells in the data defining the map of bicategories $(F,f)$, and
  any such pasting is unique due to the coherence conditions on $(F,f)$. As a
  consequence, the relations imposed in degree 3 for $y^F_f$ are also
  satisfied, hence $y^F$ is indeed an element of
  $\wCat^2(\Phi(\mathbb{A}),\Phi(\mathbb{B}))\subseteq \wCat^2_{\ccor_1}$.

  This far we have constructed a map of sets $\Phi$:
  \[
  \xymatrix{u\biCat(\mathbb{A},\mathbb{B})\ar[r]
    &\wCat^2(\Phi(\mathbb{A}),\Phi(\mathbb{B}))\ar[r] & ho
    (i^*(\wCat^2))(\Phi(\mathbb{A}),\Phi(\mathbb{B}))}
  \]
  that is clearly functorial. We still need to show that $\Phi$ is surjective
  and injective.

  To treat surjectivity, for any $y\in
  \wCat^2(\Psi(\mathbb{A}),\Psi(\mathbb{B}))\subseteq\wCat^2_{\ccor_1}$ we
  construct a homomorphism of unbiased bicategories $(F,f)\colon \mathbb{A}\To
  \mathbb{B}$ such that the associated $y^F$ described above will be in the
  class of $y$ in the homotopy category. For any such $y$ it is
  straightforward how to get the map of sets $f\colon A\To B$ and the functors
  $F_{(a_1,a_2)}\colon \mathcal{A}(a_1,a_2)\To \mathcal{B}(f(a_1),f(a_2))$,
  hence we only need to construct the natural isomorphisms for the data
  defining $(F,f)$, displayed in diagram (\ref{ubicathom}). We will discuss
  only the case $n=2$, the general case can be treated analogously. These
  natural isomorphisms are obtained as the composite of two natural
  isomorphisms given by the degree 2- and 3 data in $y_f$:
  \begin{itemize*}
  \item[(a)] The first natural isomorphism is the one in in
    $(y_f)_{\ccor_1\circ\ccor_2}$:
    \[
    \xymatrix{\mathcal{A}(a_1,a_2)\times \mathcal{A}(a_2,a_3)\ar[r]^-{\Psi^A}\ar[rd]_-{K} & \mathcal{A}(a_1,a_3)\ar[d]^{F_{13}}\\
      \ar@{}|>>>{\Uparrow\alpha}[ru]&\mathcal{B}(f(a_1),f(a_2))}
    \]
  \item[(b)] The second natural isomorphism comes from a pasting diagram,
    induced by the data in $(y_f)_{\ccor_2\circ(\ccor_1,\ccor_1)}$:
    \[
    \xymatrix@C-6mm@R+4mm{&&\mathcal{B}\times\mathcal{B}\ar[dd]^{\Psi^B}\\
      &\mathcal{B}(f(a_1),f(a_2))\times \mathcal{A}(a_2,a_3)\ar[ru]^-{\id\times F_{23}}\ar@{.>}[rd]^-H&\ar@{}[l]|-{\Uparrow\alpha_2}\\
      \mathcal{A}(a_1,a_2)\times \mathcal{A}(a_2,a_3)\ar[ru]^-{F_{12}\times \id}\ar[rr]^-K\ar[rd]_-{\id\times F_{23}}&\ar@{}[u]|{\Uparrow\beta_1}\ar@{}[d]|{\Downarrow\beta_2}&\mathcal{B}(f(a_1),f(a_3))\\
      &\mathcal{A}(a_1,a_2)\times \mathcal{B}(f(a_2),f(a_3))\ar@{.>}[ru]^-G\ar[rd]_-{F_{12}\times \id}&\ar@{}[l]|-{\Downarrow\alpha_1}\\
      &&\mathcal{B}\times \mathcal{B}\ar[uu]_{\Psi^B}}
    \]
    Note that the coherence conditions in
    $(y_f)_{\ccor_2\circ(\ccor_1,\ccor_1)}$ imply
    $\alpha_2\cdot\beta_1=\alpha_1\cdot\beta_2$ as natural isomorphisms
    $K\Rightarrow \Psi^B\circ(F_{12}\times F_{23})$.
  \end{itemize*}
  We can set the required natural isomorphism for the data in the homomorphism
  $(F,f)$ to be $\alpha\circ(\alpha_2\cdot\beta_1)$. The coherence conditions
  in the degree 3 components of $y_f$ imply that
  $(F,f)\colon\mathbb{A}\To\mathbb{B}$ is indeed a homomorphism of unbiased
  bicategories (with strict unit). The associated $y^F\in
  \wCat^2(\Phi(\mathbb{A}),\Phi(\mathbb{B}))$ is homotopic to $y$ since we can
  construct an element in $\wCat^2_{\ccor_1\circ\ccor_1}$ with faces
  $id_{\Phi(\mathbb{A})}, y$ and $y^F$. Hence the function $\Phi\colon
  u\biCat(\mathbb{A},\mathbb{B})\To
  ho(i^*\wCat^2)(\Phi(\mathbb{A},\mathbb{B})$ is surjective as well. This
  construction shows that $\Phi$ is injective as well and the proof is
  finished.
\end{proof}
Propositions \ref{bicatcomp0} and \ref{bicatcomp1} imply immediately
\begin{thm}\label{bicatmain}
  The categories $u\biCat$ and $ho(i^*(\wCat^2))$ are isomorphic. Hence the
  category of classical bicategories is equivalent to the category of
  dendroidal weak 2-categories.
\end{thm}
To conclude this Section, we conjecture that the following, stronger statement
is true:
\begin{conj}
  The inclusion of simplicial sets $N(u\biCat)\To i^*(\wCat^2)$ is a weak
  equivalence in the Joyal model structure on $s\Sets$.
\end{conj}
\appendix
\section{Notions of bicategories} \label{section:bicat} The notion of
bicategory first appeared explicitly in the paper of B\'enabou
\cite{benabou}. Intuitively, bicategories are generalised categories where the
composition of arrows is not strictly associative, only up to some coherent
2-cells which are part of the structure. The theory of bicategories had a
quick development, due to the usefulness of the notion in different
fashionable areas of mathematics. Amongst these areas we can find ordinary
category theory: as Ross Street states in \cite{street}, many fundamental
constructions of categories are bicategorical in nature. Bicategories can be
considered as generalisations of monoidal categories as well, giving new
insight to the theory of monoidal categories. Another area where bicategories
were influential is algebraic topology, especially higher homotopy theory:
bicategories are the first step in the build-up of higher categories and
groupoids, which should provide algebraic models of homotopy $n$-types.
\subsection{Classical bicategories}
A bicategory $\mathbb{A}$ consists of the following data and axioms:
\begin{itemize*}
\item[(D1)] a set $A$, called the set of objects or 0-cells;
\item[(D2)] for every ordered pair of objects $(a_1,a_2)\in A\times A$ a
  category $\mathcal{A}(a_1,a_2)$. The objects of such a category are called
  arrows or 1-cells of $\mathbb{A}$, the maps are called 2-cells of
  $\mathbb{A}$. If $f,g\in \mathcal{A}(a_1,a_2)$ are 1-cells and
  $\phi\in\mathcal{A}(a_1,a_2)(f,g)$ is a 2-cell between them, we usually
  depict this situation as $f\stackrel{\phi}{\Longrightarrow} g$ or as
  
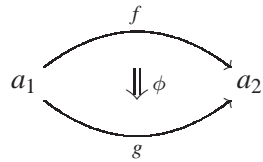
\[
  \xy (-15,0)*+{a_1}="4"; (15,0)*+{a_2}="6"; {\ar@/^1.65pc/^{f} "4";"6"};
  {\ar@/_1.65pc/_{g} "4";"6"}; {\ar@{=>}^<<<{\ \scriptstyle\phi}
    (0,2)*{};(0,-2)*{}} ;
  \endxy
  \]
  The composition of 2-cells in a category $\mathcal{A}(a_1,a_2)$ is called
  vertical composition and for two composable 2-cells $\phi, \phi'$ the
  composite is denoted by juxtaposition: $\phi\phi'$.
\item[(D3)] functors which define horizontal composition and units in
  $\mathbb{A}$:
  \begin{itemize*}
  \item [(D3a)] for all $(a_1,a_2,a_3)\in A^3$,
    $\psi\colon\mathcal{A}(a_1,a_2)\times\mathcal{A}(a_2,a_3)\To
    \mathcal{A}(a_1,a_3)$. We denote by ``$\cdot$'' the horizontal composite
    of two 1-cells (and two 2-cells), thus $g\cdot f:=\Psi(f,g)$ etc.
  \item [(D3b)] for all $a\in A$, $\psi_0\colon *\To \mathcal{A}(a,a)$, that
    is a 1-cell $\Id_a\in\mathcal{A}(a,a)$.
  \end{itemize*}
\item[(D4a)] natural isomorphisms, relating the two different ways of
  horizontal compositions of three 1-cells in $\mathbb{A}$: for all
  $a_1,a_2,a_3,a_4\in A$
  \[
  \xymatrix@R+3mm@C+8mm{\mathcal{A}(a_1,a_2)\times \mathcal{A}(a_2,a_3)\times \mathcal{A}(a_3,a_4)\ar[r]^-{\id\times \psi}\ar[d]_{\psi\times\id} & \mathcal{A}(a_1,a_2)\times \mathcal{A}(a_2,a_4)\ar[d]^\psi\\
    \mathcal{A}(a_1,a_3)\times
    \mathcal{A}(a_3,a_4)\ar@{}[ru]_-{\Uparrow\alpha}\ar[r]^-\psi
    &\mathcal{A}(a_1,a_4)}
  \]
  that is, invertible 2-cells $(h\cdot g)\cdot
  f\stackrel{\alpha}{\Longrightarrow}h\cdot (g\cdot f)$ in $\mathbb{A}$ for
  any three composable 1-cells $f,g,h$.
\item[(D4b)] natural isomorphisms, relating composition with units to the
  identity: for all $a_1,a_2\in A$,
  \[
  \xymatrix@C+3mm@R+3mm{\mathcal{A}(a_1,a_2)\times *\ar[r]^-{\id\times\psi_0}\ar@{=}[dr]&\mathcal{A}(a_1,a_2)\times \mathcal{A}(a_2,a_2)\ar[d]^-\psi\ar@{}[ld]^<<<<{\stackrel{\lambda}{\Leftarrow}}\\
    &\mathcal{A}(a_1,a_2)}
  \]
  \[
  \xymatrix@C+3mm@R+3mm{{*}\times\mathcal{A}(a_1,a_2)\ar[r]^-{\psi_0\times\id}\ar@{=}[dr]&\mathcal{A}(a_1,a_1)\times \mathcal{A}(a_1,a_2)\ar[d]^\psi\ar@{}[ld]^<<<<{\stackrel{\rho}{\Leftarrow}}\\
    &\mathcal{A}(a_1,a_2)}
  \]
  that is, for any 1-cell $f\in \mathcal{A}(a_1,a_2)$ invertible 2-cells
  \[
  \Id_{a_2}\cdot f\stackrel{\lambda}{\Longrightarrow} f \qquad\text{and}\qquad
  f\cdot \Id_{a_1}\stackrel{\rho}{\Longrightarrow} f.
  \]
\end{itemize*}
The data given above is subject to two axioms that ensure that the various
associativity and unit constraints $\alpha,\rho,\lambda$ compose coherently:
\begin{itemize*}
\item[(A1)] The following pentagon commutes for any involved composable
  1-cells
  \[
  \xymatrix@R+3mm@C-15mm{&((k\cdot h)\cdot g)\cdot f\ar@{=>}[rr]^-{\alpha\cdot\id_f}\ar@{=>}[dl]_-{\alpha}&&(k\cdot (h\cdot g))\cdot f\ar@{=>}[dr]^-{\alpha}\\
    (k\cdot h)\cdot (g\cdot f)\ar@{=>}[drr]_-{\alpha}&&&&k\cdot ((h\cdot g)\cdot f)\ar@{=>}[dll]^-{\id_k\cdot\alpha}\\
    &&k\cdot (h\cdot (g\cdot f))}
  \]
\item[(A2)] The following triangle commutes for any involved composable
  1-cells
  \[
  \xymatrix{(g\cdot \Id)\cdot f\ar@{=>}[rr]^-\alpha\ar@{=>}[dr]_-{\rho\cdot \id_f}&&g\cdot(\Id \cdot f)\ar@{=>}[ld]^-{\id_g\cdot \lambda}\\
    &g\cdot f}
  \]
\end{itemize*}
\begin{exa}
  Any (strict) 2-category is a bicategory where the associativity and unit
  2-cells $\alpha, \rho,\lambda$ are all identities.
\end{exa}
\begin{exa}
  Any monoidal category $\mathcal{C}$ is a bicategory with one 0-cell. The
  1-cells of this bicategory are the objects of $\mathcal{C}$ and the 2-cells
  are the arrows of $\mathcal{C}$. The other data and axioms of the bicategory
  are induced by the monoidal structure on $\mathcal{C}$, in the obvious way.
\end{exa}
\begin{exa}
  There exists a bicategory $\mathbb{B}\text{i}\mathbb{M}\text{od}$, defined
  as follows:
  \begin{itemize*}
  \item[--] The 0-cells of $\mathbb{B}\text{i}\mathbb{M}\text{od}$ are rings
    with unit $A, B,$ $\ldots$
  \item[--] The category of $(A,B)$-bimodules defines the 1- and 2-cells of
    $\mathbb{B}\text{i}\mathbb{M}\text{od}$.
  \item[--] Horizontal composition, units, etc. are given by tensor product of
    bimodules.
  \end{itemize*}
\end{exa}
\subsection{Homomorphisms of classical bicategories}
There exist a number of notions of homomorphisms of bicategories, the one we
define here is that of weak homomorphisms in the literature.  Thus, for us a
homomorphism of bicategories $(F,f)\colon \mathbb{A}\To\mathbb{B}$ consists of
the following data and axioms:
\begin{itemize*}
\item[(D1)] A function $f\colon A\To B$ from the set of 0-cells of
  $\mathbb{A}$ to the set of 0-cells of $\mathbb{B}$.
\item[(D2)] For every ordered pair of 0-cells $(a_1,a_2)\in A^2$ a functor
  \[
  F_{a_1a_2}\colon \mathcal{A}(a_1,a_2)\To \mathcal{B}(f(a_1),f(a_2)).
  \]
\item[(D3a)] For every $(a_1,a_2,a_3)\in A^3$ natural isomorphisms relating
  horizontal compositions and $F$:
  \[
  \xymatrix@R+3mm@C+3mm{\mathcal{A}(a_1,a_2)\times\mathcal{A}(a_2,a_3)\ar[r]^-{\psi^{\mathbb{A}}}\ar[d]_{F\times F} &\mathcal{A}(a_1,a_3)\ar[d]^F\\
    \mathcal{B}(f(a_1),f(a_2))\times
    \mathcal{B}(f(a_2),f(a_3))\ar@{}[ru]|-{\Uparrow\theta}\ar[r]_-{\psi^{\mathbb{B}}}
    &\mathcal{B}(f(a_1),f(a_3))}
  \]
  that is, invertible 2-cells $F(h)\cdot
  F(g)\stackrel{\theta}{\Longrightarrow} F(h\cdot g)$ for any composable
  1-cells $h,g\in \mathbb{A}$.
\item[(D3b)] For every $a\in A$ natural isomorphisms relating units and $F$:
  \[
  \xymatrix@C+5mm{{*}\ar[r]^-{\psi_0^{\mathbb{A}}}\ar@{=}[d] &\mathcal{A}(a,a)\ar[d]^F\\
    {*}\ar@{}[ru]|{\Uparrow\theta_0}\ar[r]_-{\psi_0^{\mathbb{B}}}
    &\mathcal{B}(f(a),f(a))}
  \]
  that is, invertible 2-cells
  $\Id^{\mathbb{B}}_{f(a)}\stackrel{\theta_0}{\Longrightarrow}F(\Id^{\mathbb{A}}_a)$
  for any $a\in A$.
\end{itemize*}
The data described above is subject to axioms that ensure that $F$ is coherent
with the various associativity- and unit-constraints:
\begin{itemize*}
\item[(A1)] For every composable 1-cells $k,h,g\in\mathbb{A}$, the following
  hexagon of invertible 2-cells commutes
  \[
  \xymatrix{&F(k\cdot h)\cdot Fg\ar@{=>}[rd]^-{\theta}\\
    (Fk\cdot Fh)\cdot Fg\ar@{=>}[d]_-{\alpha^{\mathbb{B}}}\ar@{=>}[ur]^-{\theta\cdot\id}&&F((k\cdot h)\cdot g)\ar@{=>}^-{F\alpha^{\mathbb{A}}}[d]\\
    Fk\cdot (Fh\cdot Fg)\ar@{=>}[dr]_-{\id\cdot\theta}&& F(k\cdot (h\cdot g))\\
    & Fk\cdot F(h\cdot g)\ar@{=>}[ur]_-{\theta}}
  \]
\item[(A2)] For any 1-cell $g\in \mathcal{A}(a,a')$ the following diagrams of
  invertible 2-cells commute:
  \[
  \xymatrix@C+3mm{Fg\cdot \Id^{\mathbb{B}}_{fa}\ar@{=>}[r]^-{\id\cdot\theta_0}\ar@{=>}[d]_{\rho^{\mathbb{B}}} &Fg\cdot F(\Id^{\mathbb{A}}_a)\ar@{=>}[d]^\theta &\Id^{\mathbb{B}}_{fa'}\cdot Fg\ar@{=>}[r]^-{\theta_0\cdot\id} \ar@{=>}[d]_{\lambda^{\mathbb{B}}}& F(\Id^{\mathbb{A}}_{a'})\cdot Fg\ar@{=>}[d]^\theta\\
    Fg &F(g\cdot
    \Id^{\mathbb{A}}_a)\ar@{=>}[l]^-{F(\rho^{\mathbb{A}})}&Fg&F(\Id^{\mathbb{A}}_{a'}\cdot
    g)\ar@{=>}[l]^-{F(\lambda^{\mathbb{A}})}}
  \]
\end{itemize*}
Classical bicategories and their homomorphisms form a category that we denote
by $\biCat$.
\subsection{Unbiased bicategories}
As we mentioned in the introductory part of Subsection \ref{section:bicat}, it
is more natural to compare dendroidal bicategories (the lower degree terms of
the dendroidal set $\wCat^2$) with the category of unbiased bicategories and
their homomorphisms, notions that were defined by Tom Leinster in
\cite{leinster}. We will briefly discuss them here, the resulting category of
unbiased bicategories will be denoted by $\overline{u\biCat}$. Since the
categories $\overline{u\biCat}$ and $\biCat$ are equivalent, it is justified
to compare unbiased bicategories instead of the classical ones with dendroidal
bicategories.

The idea of unbiased bicategories comes from the observation that the
definition of bicategories is ``biased'' towards a binary horizontal
composition of 1-cells and a chosen associator between the two different ways
to compose horizontally three 1-cells. One can eliminate this bias by
considering a definition which resembles operads, as follows:
\begin{itemize*}
\item[(a)] for every $n\in \N$ give a horizontal composition of (composable)
  $n$-tuples of 1-cells;
\item[(b)] relate the $n$-ary compositions for various $n$-s by some given
  2-cells (the associators);
\item[(c)] the associators should be coherent, thus they have to satisfy some
  obvious relations;
\item[(d)] take care of the unit 1-cells.
\end{itemize*}
When one tries to work out the details of the points given above, one notices
that step (b) can be fulfilled in two ways, depending on the preferred
``operadic'' approach one takes: the $\circ_i$-approach or the general
$\gamma$-approach. These definitions are equivalent (the two resulting
categories of unbiased bicategories are isomorphic). Leinster in his
definition takes the second approach, we will take here the first one:

An unbiased bicategory $\mathbb{A}$ consists of the following data:
\begin{itemize*}
\item[(D1)] a set $A$;
\item[(D2)] for every $(a_1,a_2)\in A^2$ a category $\mathcal{A}(a_1,a_2)$;
\item[(D3)] for every integer $n\ge 0$ and every sequence of objects
  $(a_1,a_2,\ldots,a_{n+1})\in A^{n+1}$ an associated functor of $n$-ary
  composition
  \[
  \xymatrix@1{\mathcal{A}(a_1,a_2)\times\mathcal{A}(a_2,a_3)\times\cdots\times\mathcal{A}(a_{n},a_{n+1})\ar[r]^-{\Psi}
    &\mathcal{A}(a_1,a_{n+1})},
  \]
  we usually denote the $n$-fold horizontal composition of 1-cells by
  \[
  (g_1\cdot g_2\cdot\ldots\cdot g_n):=\Psi(g_1,\ldots,g_n);
  \]
\item[(D4)] for all $n,m,i\in\N$ such that $1\le i\le n$, $n\ne0$ and
  composable sequences of 1-cells
  \[
  (h_1, h_2,\ldots, h_n) \textnormal{ and } (g_1,g_2,\ldots,g_m)
  \]
  such that $(g_1\cdot g_2\cdot\ldots\cdot g_m)=h_i$, natural invertible
  2-cells
  \[
  \xymatrix@1{(h_1\cdot h_2\cdot\ldots\cdot h_n)\ar@{=>}[r]^-\phi&
    (h_1\cdot\ldots\cdot h_{i-1}\cdot g_1\cdot g_2\cdot\ldots\cdot g_m\cdot
    h_{i+1}\cdot\ldots \cdot h_n)};
  \]
\item[(D5)] for every 1-cell $g$ an invertible 2-cell
  $\xymatrix@1{g\ar@{=>}[r]^-\iota &(g)}$.
\end{itemize*}
This data has to satisfy some obvious axioms, ensuring coherence of
compositions and units.

\subsection{Homomorphisms of unbiased bicategories}
Suppose that $\mathbb{A}$ and $\mathbb{B}$ are unbiased bicategories. A
homomorphism $\xymatrix@1{\mathbb{A}\ar[r]^-{(F,f)}&\mathbb{B}}$ of unbiased
bicategories consists of the following data and axioms:
\begin{itemize*}
\item[(D1)] a function $f\colon A\To B$ between the 0-cells of $\mathbb{A}$
  and $\mathbb{B}$;
\item[(D2)] for every ordered pair of 0-cells $(a_1,a_2)\in A^2$ a functor
  \[
  F_{a_1a_2}\colon \mathcal{A}(a_1,a_2)\To \mathcal{B}(f(a_1),f(a_2));
  \]
\item[(D3)] for every $n\in\N$ and every ordered sequence of 0-cells of
  $\mathbb{A}$, $(a_1,\ldots,a_n)\in A^n$ natural isomorphisms
  \begin{equation}\label{ubicathom}
    \xymatrix@R+3mm@C+5mm{\mathcal{A}(a_1,a_2)\times\cdots\times\mathcal{A}(a_{n-1},a_n)\ar[r]^-{\psi^{\mathbb{A}}}\ar[d]_{F^n} &\mathcal{A}(a_1,a_n)\ar[d]^F\\
      \mathcal{B}(f(a_1),f(a_2))\times\cdots\times \mathcal{B}(f(a_{n-1}),f(a_n))\ar@{}[ru]|-{\Uparrow\theta}\ar[r]_-{\psi^{\mathbb{B}}} &\mathcal{B}(f(a_1),f(a_n))}
  \end{equation}
\end{itemize*}
This data again is subject to some coherence axioms, ensuring the
compatibility of $F$ with the various associativity- and unit constraints of
the involved bicategories.
\begin{rem}
  The category of unbiased bicategories defined above is denoted by
  $\overline{u\biCat}$. It has a full subcategory ${u\biCat}$, whose objects
  are those unbiased bicategories for which the unit 2-cells of (D5) are all
  identities. We call them unbiased bicategories with strict unit, but note
  that this terminology is misleading since there is a chain of fully faithful
  embeddings of equivalent categories
  \[
  \biCat\subseteq u\biCat\subseteq \overline{u\biCat}.
  \]
\end{rem}

\noindent{\bf Acknowledgements.} I would like to thank Ieke Moerdijk and Ittay
Weiss for their support, suggestions and helpful comments on earlier versions
of this paper. This work was possible with the financial support of the Sectoral Operational Program for
Human Resources Development 2007-2013, co-financed by the European Social Fund, within
the project POSDRU 89/1.5/S/60189 with the title ``Postdoctoral Programs for Sustainable
Development in a Knowledge Based Society''.
\bibliographystyle{gtart}
\bibliography{/media/minden/Work/Math/myreferences/myreferences}

\def\cprime{$'$}
\begin{thebibliography}{}
\providecommand\bibmarginpar{\leavevmode\marginpar}
\def\urlstyle#1{{\tt #1}}

\bibitem{benabou}
\textbf{J B\'enabou}, \emph{Introduction to bicategories}, volume~47 of
  \emph{Lecture Notes in Mathematics}, Springer (1967)

\bibitem{berger}
\textbf{C Berger}, \emph{Double loop spaces, braided monoidal categories and
  algebraic 3-type of space}, Contemp. Math. 227 (1999)

\bibitem{bergermoerdijk2}
\textbf{C Berger}, \textbf{I Moerdijk}, \emph{The {B}oardman-{V}ogt resolution
  of operads in monoidal model categories}, Topology 45 (2006) 807--849

\bibitem{boardmanvogt}
\textbf{J\,M Boardman}, \textbf{R\,M Vogt}, \emph{Homotopy invariant algebraic
  structures on topological spaces}, volume 347 of \emph{Lecture Notes in
  Mathematics}, Springer-Verlag, Berlin (1973)

\bibitem{cisinski}
\textbf{D\,C Cisinski}, \textbf{I Moerdijk}, \emph{Dendroidal sets as models
  for homotopy operads}, Journal of Topology 4 (2011) 257--299

\bibitem{joyal}
\textbf{A Joyal}, \emph{Quasi-categories and {K}an complexes}, J. Pure Appl.
  Algebra 175 (2002)

\bibitem{joyalt}
\textbf{A Joyal}, \textbf{M Tierney}, \emph{Algebraic homotopy types}, preprint

\bibitem{leinster}
\textbf{T Leinster}, \emph{Higher operads, higher categories}, volume 298 of
  \emph{London Mathematical Society Lecture Note Series}, Cambridge University
  Press, Cambridge (2004)

\bibitem{leroy}
\textbf{O Leroy}, \emph{Sur une notion de 3-cat\'egorie adapt\'ee a
  l'homotopie}, prepub. AGATA, Univ. Montpellier II  (1994)

\bibitem{whitemac}
\textbf{S Mac~Lane}, \textbf{J\,H\,C Whitehead}, \emph{On the 3-type of a
  complex}, Proc. Nat. Acad. Sci. USA 37 (1950) 41--48

\bibitem{moerdijkweiss1}
\textbf{I Moerdijk}, \textbf{I Weiss}, \emph{Dendroidal sets}, Algebr. Geom.
  Topol. 7 (2007) 1441--1470

\bibitem{moerdijkweiss2}
\textbf{I Moerdijk}, \textbf{I Weiss}, \emph{On inner {K}an complexes in the
  category of dendroidal sets}, Adv. Math. 221 (2009) 343--389

\bibitem{street}
\textbf{R Street}, \emph{Fibrations in bicategories}, Cahiers Topologie G\'eom.
  Diff\'erentielle 21 (1980) 111--160

\bibitem{weiss}
\textbf{I Weiss}, \emph{Dendroidal sets} (2007){P}h{D}. thesis, University of
  {U}trecht

\bibitem{whitehead}
\textbf{J\,H\,C Whitehead}, \emph{Combinatorial homotopy II}, Bull. Amer. Math.
  Soc. 55 (1949) 453--496

\end{thebibliography}
\end{document}